\documentclass[12pt]{amsart}%
\usepackage{graphicx}
\usepackage{amscd}
\usepackage{geometry}
\usepackage{tikz-cd}
\usepackage{amsmath}
\usepackage{amsfonts}
\usepackage{amssymb}%
\providecommand{\U}[1]{\protect\rule{.1in}{.1in}}
\newtheorem{theorem}{Theorem}[section]
\theoremstyle{plain}

\newtheorem{corollary}[theorem]{Corollary}

\newtheorem{lemma}[theorem]{Lemma}
\newtheorem{notation}{Notation}

\newtheorem{proposition}[theorem]{Proposition}

\theoremstyle{definition}

\newtheorem{remark}[theorem]{Remark}
\newtheorem{example}[theorem]{Example}
\numberwithin{equation}{section}

\begin{document}
\title[A generalized theory of expansions and collapses]{A generalized theory of expansions and collapses with applications to
$\mathcal{Z}$-compactification}
\author{Craig R. Guilbault}
\address{Department of Mathematical Sciences\\
University of Wisconsin-Milwaukee, Milwaukee, WI 53201}
\email{craigg@uwm.edu}
\author{Daniel Gulbrandsen}
\address{Department of Mathematics and Computer Science \\
Hampden-Sydney College, Hampden-Sydney, VA 23943}
\email{dgulbrandsen@hsc.edu}
\thanks{This research was supported in part by Simons Foundation Grant 427244, CRG}
\date{October 30, 2023}

\begin{abstract}
We generalize the dual notions of \textquotedblleft
expansion\textquotedblright\ and \textquotedblleft collapse\textquotedblright%
\ so they can be applied to arbitrary metric spaces. We also expand the theory
to allow for infinitely many such moves. Those tools are then employed to
prove a variety of compactification theorems. We are particularly interested
in $\mathcal{Z}$-set compactifications, which play an important role in
geometric group theory and in algebraic and geometric topology.

\end{abstract}
\maketitle

In this paper, we expand upon the dual notions of \emph{expansion} and
\emph{collapse} in ways that make them more applicable to spaces encountered
in geometric group theory. In their original forms, these are combinatorial
moves. An \emph{elementary simplicial collapse} removes from a simplicial
complex $K$ a simplex and one of its faces in a way that results in a homotopy
equivalent subcomplex $L$. The reverse is an \emph{elementary simplicial
expansion}. A \emph{collapse} of $K$ to $L$ is a finite sequence of elementary
collapses; an \emph{expansion} is the same process, viewed in reverse. These
definitions generalize to polyhedral and CW complexes where they form the
foundation of \emph{simple homotopy theory}, which plays a central role in
algebraic and geometric topology. See, for example, \cite{Whi39}%
,\cite{Whi41},\cite{Whi49},\cite{Whi50},\cite{Coh73} and \cite{RoSa82}.

The primary objectives of this paper are as follows:

\begin{itemize}
\item[\textbf{i)}] extend the classical notions of \emph{expansion} and
\emph{collapse} to allow for \emph{infinite} sequences of such moves,

\item[\textbf{ii)}] generalize the definition of an \emph{elementary
expansion/collapse} so that it can be applied to a broader class of
topological spaces, and

\item[\textbf{iii)}] use i) and ii) to obtain compactification theorems which
allow us to add \textquotedblleft nice boundaries\textquotedblright\ to a
variety of spaces.
\end{itemize}

\noindent We are particularly interested in obtaining $\mathcal{Z}%
$-compactifications, which play an important role in geometric and algebraic
topology and geometric group theory. Here is a sample theorem.

\begin{theorem}
Every (infinitely) collapsible locally finite simplicial, cubical, or CW
complex is $\mathcal{Z}$-compactifiable.
\end{theorem}

A concrete application can be found in the second author's dissertation
\cite{Gul23}, where it is shown that every locally finite CAT(0) cube complex
$X$ is collapsible. The resulting $\mathcal{Z}$-boundary (referred to as the
\emph{cubical boundary}) shares properties with the visual boundary, but
retains some of the combinatorial structure of $X$. It can be viewed as a
hybrid of the visual and Roller boundaries.

None of the goals listed above are entirely new. Allowing an \emph{infinite}
sequence expansions or collapses is a fairly obvious idea, but its appearance
in the literature is surprisingly rare. One exception is a recent arXiv
posting by Adiprasito and Funar and \cite{AdFu21}. A different approach, was
developed by Siebenmann \cite{Sie70a}, under which only finitely many moves
are permitted, but a move can consist of an infinite discrete collection of
\emph{simultaneous} elementary expansions or collapses. The resulting theory,
known as \emph{infinite simple homotopy theory}, is important but different
from what we do here. For example, the move just described can only be
performed on a noncompact space. By contrast, our approach can \emph{end} with
a noncompact space, but at each step along the way the spaces are typically compact.

To accomplish the second goal, we leave the world of combinatorially
constructed spaces to define a \textquotedblleft compact topological
collapse\textquotedblright\ (or simply \emph{compact collapse}) of one space
onto another. These compact collapses are the same as Ferry's
\emph{contractible retractions} \cite{Fer80}. We will comment more on that
work in Section \ref{Section: Topological expansions and collapses}.

Our main compactification theorems generalize some well known constructions.
One of those is the $\mathcal{Z}$-compactification of the mapping telescope of
an inverse sequence of finite polyhedra, which played a key role in work by
Chapman and Siebenmann \cite[\S 4]{ChSi76}. Another is the addition of the
visual boundary to a proper CAT(0) space $X$, as described by Bridson and
Haefliger \cite[\S II.8.5]{BrHa99}. They do not mention $\mathcal{Z}$-sets;
that connection, and its usefulness, was noted by Bestvina \cite[Example
1.2(ii)]{Bes96}. Yet another is the proof, as presented by Jeremy Brazus in
his blog \emph{Wild Topology}, of a theorem by Borsuk, showing that all\emph{
dendrites} are contractible. Each of these is a special case of results
presented in Sections \ref{Section: Compactifications} and
\ref{Section: Examples}.

The existence of an infinite sequence of compact expansions $C_{0}\swarrow
C_{1}\swarrow C_{3}\swarrow C_{4}\swarrow\cdots$, beginning with a compactum
$C_{0}$, leads to a inverse sequence of retractions
\[
C_{0}\overset{r_{1}}{\longleftarrow}C_{1}\overset{r_{2}}{\longleftarrow}%
C_{2}\overset{r_{3}}{\longleftarrow}\cdots
\]
a direct sequence of inclusions%
\[
C_{0}\overset{s_{0}}{\longrightarrow}C_{1}\overset{s_{1}}{\longrightarrow
}C_{2}\overset{s_{2}}{\longrightarrow}\cdots
\]
and a set $X=\cup_{i=0}^{\infty}C_{i}$, whose topology is ambiguous. Our
compactification strategy relies on the inverse limit construction, which
determines the correct (but not always the obvious) topology for $X$. This
realization led to an unexpected, but interesting, fourth major theme of this article:

\begin{itemize}
\item[\textbf{iv)}] explore the topology of inverse sequences of retractions
and corresponding direct sequences of inclusion maps.
\end{itemize}

\noindent This topic provides another connection to the literature. In
\cite{MaPr15}, Marsh and Prajs studied inverse sequences of retractions from a
different perspective. Whereas we use these sequences for constructing and
compactifying noncompact spaces, they \emph{begin} with a compact space $Y$
and look for situations when $Y$ can be realized as the inverse limit of
compact subspaces $C_{0},C_{1},C_{2},\cdots$ where the bonding maps are retractions.

The layout of this paper is as follows. Section
\ref{Section: The general setting} is general topology. We identify topologies
$\mathcal{T}_{I}$ and $\mathcal{T}_{D}$ on $X$---one is strongly associated
with the inverse sequence; the other with the direct sequence. We prove a
useful theorem that shows when they agree; then we use the inverse limit
construction to define a canonical compactification of $\left(  X,\mathcal{T}%
_{I}\right)  $ that is used throughout the paper. In Section
\ref{Section: Examples}, we bring these abstractions to life with a variety of
examples. In Section
\ref{Section: Expansions and collapses in the simplicial category} we review
the classical theory of expansions and collapses, and develop terminology and
notation that allows for infinitely many such moves. In Section
\ref{Section: Topological expansions and collapses} we develop generalizations
of the concepts in Section
\ref{Section: Expansions and collapses in the simplicial category}. In Section
\ref{Section: Compactifications} we apply all of this to prove a variety of
compactification theorems; of particular interest are $\mathcal{Z}%
$-compactifications. We conclude the paper with a variety of applications.

\section{The general setting\label{Section: The general setting}}

In the most general setting, we are interested in spaces that can be expressed
as monotone unions of compact metric spaces $C_{0}\subseteq C_{1}\subseteq
C_{2}\subseteq\cdots$ where, for each $i>0$, there is a retraction
$r_{i}:C_{i}\rightarrow C_{i-1}$. This situation is concisely described with
an an inverse sequence
\begin{equation}
C_{0}\overset{r_{1}}{\longleftarrow}C_{1}\overset{r_{2}}{\longleftarrow}%
C_{2}\overset{r_{3}}{\longleftarrow}\cdots
\label{inverse sequence of retractions}%
\end{equation}
There is a corresponding direct sequence
\begin{equation}
C_{0}\overset{s_{0}}{\longrightarrow}C_{1}\overset{s_{1}}{\longrightarrow
}C_{2}\overset{s_{2}}{\longrightarrow}\cdots
\label{direct sequence of inclusions}%
\end{equation}
where the $s_{i}$ are inclusions. There is also a set $X=\cup_{i=0}^{\infty
}C_{i}$. Sequences (\ref{inverse sequence of retractions}) and
(\ref{direct sequence of inclusions}) suggest preferred topologies for $X$,
but those topologies may not agree. We will return to that issue after
establishing some notation.

For $i\leq j$, let $r_{ij}:C_{i}\leftarrow C_{j}$ denote the obvious
composition, with special cases $r_{ii}=\operatorname*{id}_{C_{i}}$ and
$r_{(i-1)i}=r_{i}$. Note that: each $r_{ij}$ is a retraction; $r_{ij}\circ
r_{jk}=r_{ik}$; and $\left.  r_{ik}\right\vert _{C_{j}}=r_{ij}$ whenever
$i\leq j\leq k$. Thus, we may define retractions (on the level of sets)
$r_{i\infty}:C_{i}\leftarrow X$ by letting $\left.  r_{i\infty}\right\vert
_{C_{j}}=r_{ij}$. Similarly, let $s_{ij}:C_{i}\rightarrow C_{j}$ denote the
inclusion map or, equivalently, the composition of maps from
(\ref{direct sequence of inclusions}). Let $s_{i\infty}:C_{i}\rightarrow X$
also denote inclusion.

One natural way to topologize $X$ is with the \emph{coherent topology} with
respect to the collection $\left\{  C_{i}\right\}  _{i=0}^{\infty}$. Under
that topology, $U\subseteq X$ is open if and only if $U\cap C_{i}$ is open in
$C_{i}$ for all $i$. This is equivalent to the \emph{final topology} on $X$
with respect to the collection $\left\{  s_{i\infty}:C_{i}\rightarrow
X\right\}  _{i=0}^{\infty}$ which, by definition, is the finest topology on
$X$ under which all of these functions are continuous. We denote this topology
by $\mathcal{T}_{D}$ since it is closely related to the \emph{direct} sequence
(\ref{direct sequence of inclusions}).

An equally natural topology for $X$ is with the \emph{initial topology} with
respect to the collection $\left\{  r_{i\infty}:X\rightarrow C_{i}\right\}
_{i=0}^{\infty}$. By definition, this is the coarsest topology under which
these functions are all continuous, i.e., it is the topology generated by the
subbasis consisting of sets of the form $r_{i\infty}^{-1}\left(  U\right)  $,
where $U$ is an open subset of $C_{i}$ for some $i$. We denote this topology
by $\mathcal{T}_{I}$ since it is most naturally associated with the
\emph{inverse }sequence (\ref{inverse sequence of retractions}). In the
special cases under consideration here, this subbasis is actually a basis.

\begin{lemma}
Given the above setup, the defining subbasis
\[
\mathcal{S}_{I}=\left\{  r_{i\infty}^{-1}\left(  U\right)  \mid U\text{ is an
open subset of }C_{i}\text{ for some }i\right\}
\]
for $\mathcal{T}_{I}$ is a basis for $\mathcal{T}_{I}$.
\end{lemma}

\begin{proof}
It suffices to show that if $U$ and $V$ are open subsets of $C_{i}$ and
$C_{j}$ respectively, then $r_{i\infty}^{-1}\left(  U\right)  \cap r_{j\infty
}^{-1}\left(  V\right)  \in\mathcal{S}_{I}$. First note that, if $i=j$, then
$r_{i\infty}^{-1}\left(  U\right)  \cap r_{i\infty}^{-1}\left(  V\right)
=r_{i\infty}^{-1}\left(  U\cap V\right)  \in\mathcal{S}$. If $i<j$, then
$r_{i\infty}^{-1}\left(  U\right)  =r_{j\infty}^{-1}\left(  r_{ij}^{-1}\left(
U\right)  \right)  $, so the desired conclusion follows from the previous case.
\end{proof}

\begin{remark}
Since the inclusion maps in sequence (\ref{direct sequence of inclusions}) do
not depend upon the bonding maps in (\ref{inverse sequence of retractions}),
$\mathcal{T}_{D}$ is essentially well-defined by the collection $\left\{
C_{i}\right\}  _{i=0}^{\infty}$. By contrast, the topology $\mathcal{T}_{I}$
depends on the bonding maps in (\ref{inverse sequence of retractions}).
Example of this phenomenon can be found in Examples
\ref{Example: cone on a sequence 1}-\ref{Example: compactifications of a ray}
of the following section..
\end{remark}

A few more definitions and elementary observations will be useful as we proceed.

\begin{lemma}
\label{Lemma: general topology items}Given the above setup,

\begin{enumerate}
\item \label{Item 1 of general topology}$\mathcal{T}_{D}$ is as fine as or
finer than $\mathcal{T}_{I}$; i.e., $\mathcal{T}_{I}\subseteq\mathcal{T}_{D}$, and

\item \label{Item 2 of general topology}in both $\left(  X,\mathcal{T}%
_{D}\right)  $ and $\left(  X,\mathcal{T}_{I}\right)  $, each $C_{i}$ is a
subspace of $X$.
\end{enumerate}
\end{lemma}

\begin{proof}
To verify \ref{Item 1 of general topology}), let $U$ be an open subset of
$C_{i}$ and $r_{i\infty}^{-1}\left(  U\right)  $ be the corresponding basic
open set in $\mathcal{T}_{I}$. Since $\left.  r_{i\infty}\right\vert _{C_{i}%
}=\operatorname*{id}_{C_{i}}$, it follows that $r_{i\infty}^{-1}\left(
U\right)  \cap C_{i}=U$. So, for all $j\leq i$, $r_{i\infty}^{-1}\left(
U\right)  \cap C_{j}=U\cap C_{j}$ is open in $C_{j}$. For $j>i$, $r_{i\infty
}^{-1}\left(  U\right)  \cap C_{j}=r_{ij}^{-1}\left(  U\right)  $, which is
open in $C_{j}$ since $r_{ij}$ is continuous. Hence $r_{i\infty}^{-1}\left(
U\right)  \in\mathcal{T}_{D}$.

For \ref{Item 2 of general topology}), let $\mathcal{T}_{C_{i}}$ denote the
original topology on $C_{i}$, $\mathcal{T}_{C_{i}}^{\prime}$ the subspace
topology on $C_{i}$ induced by $\mathcal{T}_{D}$, and $\mathcal{T}_{C_{i}%
}^{^{\prime\prime}}$ the subspace topology on $C_{i}$ induced by
$\mathcal{T}_{I}$.

If $U^{\prime}\in\mathcal{T}_{C_{i}}^{\prime}$ there exists $U\in
\mathcal{T}_{D}$ such that $U^{\prime}=U\cap C_{i}$, so by definition of
$\mathcal{T}_{D}$, $U^{\prime}\in\mathcal{T}_{C_{i}}$. Conversely, if
$U^{\prime}\in\mathcal{T}_{C_{i}}$, then $r_{i\infty}^{-1}\left(  U^{\prime
}\right)  \in\mathcal{T}_{I}$ so, by \ref{Item 1 of general topology}),
$r_{i\infty}^{-1}\left(  U^{\prime}\right)  \in\mathcal{T}_{D}$. Since
$r_{i\infty}^{-1}\left(  U^{\prime}\right)  \cap C_{i}=U^{\prime}$,
$U^{\prime}\in\mathcal{T}_{C_{i}}^{\prime}$.

The latter half of the above argument assures that $\mathcal{T}_{C_{i}%
}\subseteq\mathcal{T}_{C_{i}}^{^{\prime\prime}}$. For the reverse inclusion,
let $r_{j\infty}^{-1}\left(  U\right)  $ be a basic open subset in
$\mathcal{T}_{I}$ where $U\in\mathcal{T}_{C_{j}}$ and $U^{\prime\prime
}=r_{j\infty}^{-1}\left(  U\right)  \cap C_{i}\in\mathcal{T}_{C_{i}}%
^{^{\prime\prime}}$. By again following the argument in
\ref{Item 1 of general topology}), we know that $r_{j\infty}^{-1}\left(
U\right)  \cap C_{i}$ is either $U\cap C_{i}$ or $r_{ij}^{-1}\left(  U\right)
$, depending on whether $i\leq j$ or $i>j$. In both cases we have an element
of $\mathcal{T}_{C_{i}}$.
\end{proof}

\begin{notation}
Going forward, we will often simplify notation by letting $X_{D}$ and $X_{I}$
denote $\left(  X,\mathcal{T}_{D}\right)  $ and $\left(  X,\mathcal{T}%
_{I}\right)  $. Similarly, when $A\subseteq X$, we will let
$\operatorname*{Int}_{D}A$ and $\operatorname*{Int}_{I}A$ to denote the
interiors of $A$ in those spaces. Lemma \ref{Lemma: general topology items}
allows us to view $C_{i}$ as a subspace of $X_{D}$ or $X_{I}$, without need
for specification.
\end{notation}

We say that $A\subseteq X$ is $j$\emph{-stationary }(with respect to the
sequence $\left\{  C_{i},r_{i}\right\}  $) if $A\subseteq C_{j}$ and
$r_{j\infty}^{-1}\left(  A\right)  =A$. The following is immediate.

\begin{lemma}
\label{Lemma: j-stationary properties}Given the above setup,

\begin{enumerate}
\item every subset of a $j$-stationary set is $j$-stationary,

\item every union of $j$-stationary sets is $j$-stationary, and

\item if $A\subseteq X$ is $j$-stationary, then $r_{jk}^{-1}\left(  A\right)
=A$ and $r_{k\infty}^{-1}\left(  A\right)  =A$ for all $k\geq j$. In
particular, $A$ is $k$-stationary, for all $k\geq j$.
\end{enumerate}
\end{lemma}

We say that $A\subseteq X$ is $j$-\emph{insulated} if there exists a
$j$-stationary open neighborhood of $A$ in $C_{j}$. Lemma
\ref{Lemma: j-stationary properties} gives the following.

\begin{lemma}
If $A\subseteq X$ is $j$-insulated, then $A$ is $k$-insulated for all $k\geq
j$.
\end{lemma}

We will say that $A\subseteq X$ is \emph{stationary} [\emph{insulated}] if
there is some $j\in%
\mathbb{N}
$ for which $A$ is $j$-stationary [$j$-insulated]. If a set $\left\{
x\right\}  $ has one of these properties, we attribute that property to the
point $x$.

\begin{lemma}
\label{Lemma: x-insulated implies C_i insulated}If each $x\in C_{i}$ is
insulated, then $C_{i}$ is insulated.
\end{lemma}

\begin{proof}
By hypothesis, for each $x\in C_{i}$, there exists a $j_{x}\geq0$ and a
$j_{x}$-stable open neighborhood $U_{j_{x}}$ of $x$ in $C_{j_{x}}$. When
$j_{x}\leq i$, Lemma \ref{Lemma: j-stationary properties} and continuity imply
that $U_{j_{x}}$ is open in $C_{i}$ whereas, when $j_{x}>i$, $U_{j_{x}}\cap
C_{i}$ is open in $C_{i}$. So, by compactness, there exists a finite sequence
$j_{x_{1}}\leq\cdots\leq j_{x_{k}}$ such that the sets $U_{j_{x_{1}}}%
,\cdots,U_{j_{x_{k}}}$ cover $C_{i}$. If $j_{x_{k}}\leq j$, then $C_{i}$ is a
union of $j$-stationary open subsets $U_{j_{x_{1}}},\cdots,U_{j_{x_{k}}}$, and
is therefore $j$-insulated. If $j_{x_{k}}>i$, then each of $U_{j_{x_{1}}%
},\cdots,U_{j_{x_{r}}}$ is $j_{x_{k}}$-stationary and open in $C_{j_{x_{k}}}$.
Since $C_{i}\subseteq\cup_{r=1}^{k}U_{j_{x_{r}}}$ it follows that $C_{i}$ is
$j_{x_{k}}$-insulated.
\end{proof}

\begin{theorem}
\label{Theorem: TFAE with T_I = T_D}Given the above setup, the following are equivalent.

\begin{enumerate}
\item \label{Item 1 of TFAE}Each point of $X$ is insulated.

\item \label{Item 2 of TFAE}$X=\cup_{i=0}^{\infty}\operatorname*{Int}_{I}%
C_{i}$.

\item \label{Item 3 of TFAE}$\mathcal{T}_{I}=\mathcal{T}_{D}$.
\end{enumerate}

\begin{proof}
We begin by proving \ref{Item 1 of TFAE})$\Leftrightarrow$\ref{Item 2 of TFAE}%
). For the forward implication, note that Lemma
\ref{Lemma: x-insulated implies C_i insulated} guarantees that for each $i$,
there exists an integer $k>i$ and an open subset $U_{k}\subseteq C_{k}$ that
contains $C_{i}$ and for which $r_{k\infty}^{-1}\left(  U_{k}\right)  =U_{k}$.
This implies that $U_{k}$ is a basic open set in $\mathcal{T}_{I}$, so
$C_{i}\subseteq\operatorname*{Int}_{I}C_{k}$. The desired conclusion follows.

For the reverse implication, let $x\in X$. Then there exists $j\geq0$ such
that $x\in\operatorname*{Int}_{I}C_{j}$, so there exists a basic open
neighborhood $r_{k\infty}^{-1}\left(  U_{k}\right)  $ of $x$ contained in
$C_{j}$. If $k\geq j$, then $r_{k\infty}^{-1}\left(  U_{k}\right)  \subseteq
C_{k}$, so $x$ is $k$-insulated. If $k<j$, then $r_{kj}^{-1}\left(
U_{k}\right)  $ is open in $C_{j}$ and $r_{j\infty}^{-1}\left(  r_{kj}%
^{-1}\left(  U_{k}\right)  \right)  =r_{k\infty}^{-1}\left(  U_{k}\right)
\subseteq C_{j}$, so $x$ is $j$-insulated.

Next we show \ref{Item 2 of TFAE})$\Rightarrow$\ref{Item 3 of TFAE}). We
already know $\mathcal{T}_{I}\subseteq\mathcal{T}_{D}$, so let $V\in
\mathcal{T}_{D}$ and $x\in V$. By hypothesis, there is a $j\geq0$ and an open
neighborhood $U_{j}$ of $x$ in $C_{j}$ such that $r_{j\infty}^{-1}\left(
U_{j}\right)  =U_{j}$. Since $V\in\mathcal{T}_{D}$, $V\cap C_{j}$ is open in
$C_{j}$, so $V_{j}:=V\cap U_{j}$ is open in $C_{j}$. Moreover, $r_{j\infty
}^{-1}\left(  V_{j}\right)  =V_{j}$, so $V_{j}$ is a basic open set in
$\mathcal{T}_{I}$. Since $x\in V_{j}\subseteq V$, $V\in\mathcal{T}_{I}$.

Lastly, we use a contrapositive argument to show \ref{Item 3 of TFAE}%
)$\Rightarrow$\ref{Item 1 of TFAE}). Suppose $x\in X$ is not insulated, and
let $i_{0}$ be the minimum integer for which $x\in C_{i_{0}}$. For each $k\geq
i_{0}$, let $\mathcal{B}_{k}=\{B_{k,j}\}_{j=k}^{\infty}$ be a countable
neighborhood basis at $x$ of open subsets in $C_{k}$, and for each $k\geq
i_{0}$ and $j\geq k$, let $B_{k,j}^{\prime}=r_{k\infty}^{-1}\left(
B_{k,j}\right)  $ be the corresponding basic open neighborhood of $x$ in
$\left(  X,\mathcal{T}_{I}\right)  $. If $\mathcal{B}_{k}^{\prime}$ denotes
$\{B_{k,j}^{\prime}\}_{j=k}^{\infty}$ then $\cup_{k=i_{0}}^{\infty}%
\mathcal{B}_{k}^{\prime}$ is a neighborhood basis at $x$ in $\left(
X,\mathcal{T}_{I}\right)  $. Since $x$ is \ not insulated, for each $k\geq
i_{0}$ and $j\geq k$, we may choose $x_{kj}\in B_{k,j}^{\prime}-C_{j}$. Let
$A$ be the set of all such $x_{k,j}$ and notice that $x\notin A$ and $A\cap
C_{i}$ is finite, for all $i\geq0$. As such, $A$ is closed in $\left(
X,\mathcal{T}_{D}\right)  $. But, since every element of $\cup_{k=i_{0}%
}^{\infty}\mathcal{B}_{k}^{\prime}$ intersects $A$, $A$ is not closed in
$\left(  X,\mathcal{T}_{I}\right)  $.
\end{proof}
\end{theorem}

\begin{corollary}
If any of the conditions in Theorem \ref{Theorem: TFAE with T_I = T_D} hold,
then $X=\cup_{i=0}^{\infty}\operatorname*{int}_{D}C_{i}$. The converse of this
assertion is false.
\end{corollary}

\begin{proof}
The assertion is immediate from condition \ref{Item 3 of TFAE}), while part
iii) of Example \ref{Example: compactifications of a ray} in Section
\ref{Section: Examples} is a counterexample to the converse.
\end{proof}

When the conditions in Theorem \ref{Theorem: TFAE with T_I = T_D} hold, we say
that $\left\{  C_{i},r_{i}\right\}  $ is \emph{fully insulated}. Example
\ref{Example: tangent disks} from Section \ref{Section: Examples} shows that,
for this to be true, it is not enough to assume that each $C_{i}$ is stationary.

\subsection{Inverse and direct limits associated with
$X\label{Subsection: Inverse and direct limits associated with X}$}

We now provide further justification for associating the topologies
$\mathcal{T}_{D}$ and $\mathcal{T}_{I}$ with the direct and inverse sequences
(\ref{direct sequence of inclusions}) and
(\ref{inverse sequence of retractions}). In particular, we look at the roles
played by the limits of those sequences.

The \emph{direct limit} of (\ref{direct sequence of inclusions}) can be
defined by
\[
\underrightarrow{\lim}\left\{  C_{i},s_{i}\right\}  :=(\sqcup_{i=0}^{\infty
}C_{i}\times\left\{  i\right\}  )/\sim
\]
where $\left(  x.i\right)  \sim\left(  y,j\right)  $ if there exist $k\geq
\max\left\{  i,j\right\}  $ such that $s_{ik}\left(  x\right)  =s_{jk}\left(
y\right)  $. The \emph{direct limit topology }is the quotient topology induced
by%
\[
q:\sqcup_{i=0}^{\infty}C_{i}\times\left\{  i\right\}  \rightarrow(\sqcup
_{i=0}^{\infty}C_{i}\times\left\{  i\right\}  )/\sim
\]
where the left-hand side is given the disjoint union topology. In the case at
hand, there is a bijection $f:X\rightarrow\underrightarrow{\lim}\left\{
C_{i},s_{i}\right\}  $ taking $x\in X$ to the equivalence class $\left[
\left(  x,i\right)  \right]  $, where $i$ is chosen sufficiently large that
$x\in C_{i}$. If we use $f$ to pull back the direct limit topology, we obtain
a topology on $X$ under which $U\subseteq X$ open if and only if $U\cap C_{i}$
is open in $C_{i}$ for all $i$. In other words, we arrive at $\mathcal{T}_{D}%
$, which we think of, informally, as the direct limit topology on $X$.

There is a similar, but more subtle, relationship between $\mathcal{T}_{I}$
and the inverse limit of (\ref{inverse sequence of retractions}). Recall that%
\[
\underleftarrow{\lim}\left\{  C_{i},r_{i}\right\}  :=\left\{  \left(
x_{0},x_{1},x_{2},\cdots\right)  \mid r_{i}\left(  x_{i}\right)
=x_{i-1}\text{ for all }i>0\right\}
\]
is topologized as a subspace of $%
{\textstyle\prod_{i=0}^{\infty}}
C_{i}$, where the latter is given the product topology. There is a natural
injection $e:X\rightarrow\underleftarrow{\lim}\left\{  C_{i},r_{i}\right\}  $
defined by $e\left(  x\right)  =\left(  r_{i\infty}\left(  x\right)  \right)
_{i=0}^{\infty}$. Under both $\mathcal{T}_{D}$ and $\mathcal{T}_{I}$, the
coordinate function are continuous, so $e$ is a continuous injection.

\begin{theorem}
\label{Theorem: e is an embedding into the inverse limit}The function
$e:X\rightarrow\underleftarrow{\lim}\left\{  C_{i},r_{i}\right\}  $ is an
embedding if and only if $X$ is given the topology $\mathcal{T}_{I}$; $e$ need
not be surjective, but $e\left(  X\right)  $ is dense in $\underleftarrow{\lim
}\left\{  C_{i},r_{i}\right\}  $.
\end{theorem}

\begin{proof}
Clearly, there is a unique topology on $X$ under which $e$ is an embedding. We
show that $\mathcal{T}_{I}$ is that topology.

Let $U=r_{j\infty}^{-1}\left(  V_{j}\right)  $ be a basic open set in
$\mathcal{T}_{I}$, where $V_{j}$ is an open subset of $C_{j}$, and let
$\pi_{j}:%
{\textstyle\prod_{i=0}^{\infty}}
C_{i}\rightarrow C_{j}$ be the projection map. Then
\begin{align*}
e\left(  U\right)   &  =\left\{  e\left(  x\right)  \mid x\in X\text{ and
}r_{j\infty}\left(  x\right)  \in V_{j}\right\} \\
&  =e\left(  X\right)  \cap\pi_{j}^{-1}\left(  V_{j}\right) \\
&  =e\left(  X\right)  \cap\left(  \underleftarrow{\lim}\left\{  C_{i}%
,r_{i}\right\}  \cap\pi_{i}^{-1}\left(  V_{i}\right)  \right)
\end{align*}
Since $\pi_{i}^{-1}\left(  V_{i}\right)  $ is open in the product topology,
then $\underleftarrow{\lim}\left\{  C_{i},r_{i}\right\}  \cap\pi_{i}%
^{-1}\left(  V_{i}\right)  $ is open in $\underleftarrow{\lim}\left\{
C_{i},r_{i}\right\}  $, so $e\left(  U\right)  $ is open in $e\left(
X\right)  $. Therefore, the corestriction of $e$ to $e\left(  X\right)  $ is
an open map, hence a homeomorphism.

To see that $e\left(  X\right)  $ is dense in $\underleftarrow{\lim}\left\{
C_{i},r_{i}\right\}  $, notice that if $x\in C_{i}$ then $r_{j\infty}\left(
x\right)  =x$ for all $j\geq i$. As such, $e\left(  X\right)  $ is the set of
all eventually constant sequences $\left(  x_{0},\cdots,x_{k-1},c,c,c,\cdots
\right)  $ such that $c\in C_{k}$, $x_{k-1}=r_{k}\left(  c\right)  $, and
$x_{i}=r_{i+1}\left(  x_{i+1}\right)  $ for all $i<k-1$. Clearly, every
element of $\underleftarrow{\lim}\left\{  C_{i},r_{i}\right\}  $ is the limit
of a sequence of such points.

Non-surjectivity of $e$ occurs in most of the examples presented in Section
\ref{Section: Examples}.
\end{proof}

In most cases of interest to us, $X_{I}$ (and therefore $X_{D}$) will be
noncompact. We have assumed that each $C_{i}$ is a compact metric space, so $%
{\textstyle\prod_{i=0}^{\infty}}
C_{i}$ is compact and metrizable and $\underleftarrow{\lim}\left\{
C_{i},r_{i}\right\}  $ is a closed subset of $%
{\textstyle\prod_{i=0}^{\infty}}
C_{i}$. See, for example, \cite{Dug78}. If we adopt a standard convention,
that a \emph{compactification} of a space $S$ is a compact Hausdorff space
that contains $S$ as a dense subspace, then $\underleftarrow{\lim}\left\{
C_{i},r_{i}\right\}  $ is a compactification of $e\left(  X\right)  $. Letting
$R:=\underleftarrow{\lim}\left\{  C_{i},r_{i}\right\}  -e\left(  X\right)  $,
we have
\[
\underleftarrow{\lim}\left\{  C_{i},r_{i}\right\}  =\overline{e\left(
X\right)  }=e\left(  X\right)  \sqcup R
\]
By following a standard procedure we arrive at:

\begin{corollary}
The space $X_{I}$ is separable and metrizable and admits a compactification
$\overline{X_{I}}=X_{I}\sqcup R$ which is homeomorphic to
$\underleftarrow{\lim}\left\{  C_{i},r_{i}\right\}  $.
\end{corollary}

\begin{proof}
As a subspace of the compact metrizable space $%
{\textstyle\prod_{i=0}^{\infty}}
C_{i}$, $e\left(  X\right)  $ is separable and metrizable---so $X_{I}$ is
separable and metrizable. To obtain $\overline{X_{I}}=X_{I}\sqcup R$, extend
$e:X_{I}\rightarrow e\left(  X\right)  $ to the function $\overline
{e}:\overline{X_{I}}\rightarrow\overline{e\left(  X\right)  }$ which is the
identity on $R$. Then topologize $\overline{X}$ by declaring $U\subseteq
\overline{X_{I}}$ to be open if and only if $\overline{e}\left(  U\right)  $
is open in $\overline{e\left(  X\right)  }$.
\end{proof}

Under the most general of the circumstances described so far, $X_{I}$ and
$\overline{X_{I}}$ can be quite unusual. Since we are primarily interested in
locally compact metric spaces and their compactifications, the following
observation is relevant.

\begin{proposition}
\label{Proposition: local compactness of X}For an arbitrary compactification
$\overline{X}$ of a Hausdorff space $X$, $X$ is open in $\overline{X}$ if and
only if it is locally compact. For the compactifications $\overline{e\left(
X\right)  }=e\left(  X\right)  \sqcup R$ and $\overline{X_{I}}=X_{I}\sqcup R$
described above, these conditions hold whenever the conditions in Theorem
\ref{Theorem: TFAE with T_I = T_D} are satisfied.
\end{proposition}

\begin{proof}
For the forward implication of the initial assertion, note that $\overline{X}$
is locally compact and Hausdorff. Hence, local compactness is truly a local
property \cite[Th.29.2]{Mun00}, so every open subset of $\overline{X}$ is
locally compact. The reverse implication is well known but less obvious; see
\cite[XI.8.3]{Dug78}.

To prove the latter assertion, let $x\in X_{I}$ be arbitrary. By hypothesis,
there exists a $j\geq0$ and an open neighborhood $U_{j}$ of $x$ in $C_{j}$
such that $r_{j\infty}^{-1}\left(  U_{j}\right)  =U_{j}$. Since $C_{j}$ is
compact, it is locally compact, so $U_{j}$ contains a compact neighborhood
$V_{j}$ of $x$ in the space $C_{j}$. But $r_{j\infty}^{-1}\left(
V_{j}\right)  =V_{j}$, so $V_{j}$ is also a neighborhood of $x$ in $X_{I}$.
\end{proof}

\begin{remark}
Example \ref{Example: cone on a sequence 2} shows that the converse of the
second sentence of Proposition \ref{Proposition: local compactness of X} is
false. One might ask whether local compactness of $\left(  X,\mathcal{T}%
_{I}\right)  $ implies that $\mathcal{T}_{I}=\mathcal{T}_{D}$. The same
example shows that the answer is \textquotedblleft no\textquotedblright.
\end{remark}

\section{Examples\label{Section: Examples}}

In this section, we present a collection examples that fit the above setup and
exhibit a variety of properties. As additional topics are developed, we will
often refer back to these examples.

\begin{example}
\label{Example: cone on a sequence 1}Let $P=\left\{  p_{1},p_{2}%
,\cdots\right\}  $ be a sequence of points with the discrete topology. For
$i\geq1$, let $C_{i}$ be the cone on $\left\{  p_{1},\cdots,p_{i}\right\}  $
with cone point $q$ and let $C_{0}=\left\{  q\right\}  $. Then, as a set,
$X=\cup_{i=0}^{\infty}C_{i}$ is the cone on $P$, and we have a direct system
of inclusion maps%
\[
C_{0}\overset{s_{0}}{\longrightarrow}C_{1}\overset{s_{1}}{\longrightarrow
}C_{2}\overset{s_{2}}{\longrightarrow}\cdots
\]
For each $i>0$, let $r_{i}:C_{i-1}\leftarrow C_{i}$ be the retraction which
maps segment segment $\overline{qp}_{i}$ linearly onto $\overline{qp}_{i-1}$
by sending $p_{i}$ to $p_{i-1}$. This gives an inverse sequence
\[
C_{0}\overset{r_{1}}{\longleftarrow}C_{1}\overset{r_{2}}{\longleftarrow}%
C_{2}\overset{r_{3}}{\longleftarrow}\cdots
\]

With a little effort, one can show that $X_{D}$ does not have a countable
local basis at $q$, and hence is nonmetrizable. By contrast, $X_{I}$ is not
only metrizable, but planar. In fact, $\overline{X_{I}}=X_{I}\sqcup R$ is
homeomorphic to the left-hand compactum shown in Figure \ref{Figure 1}
with $R\approx(0,1]$ corresponding to the limiting arc minus the point $q$.
\end{example}

\begin{example}
\label{Example: cone on a sequence 2}Take the same collection $\left\{
C_{i}\right\}  $ as in the previous example, but now define $r_{i}^{\prime
}:C_{i-1}\leftarrow C_{i}$ to take the entire segment $\overline{qp}_{i}$ to
the point $q$. The direct system and $X_{D}$ remain the same, but with these
bonding maps, $X_{I}$ is homeomorphic to the planar continuum on the right in
Figure 1. In this case, $R$ is the empty set.
\end{example}

\begin{figure}[th]
\begin{center}
\includegraphics[scale=1.0]{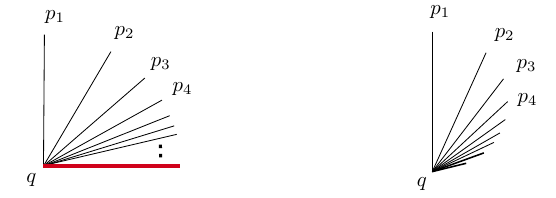} \label{Figure 1}
\end{center}
\caption{$\overline{X_{I}}=X_{I}\sqcup R$ in Examples
\ref{Example: cone on a sequence 1} (left) and
\ref{Example: cone on a sequence 2} (right)}%
\label{Figure 1}%
\end{figure}

\begin{example}
\label{Example: compactifications of a ray}For each $i\geq0$, let $C_{i}$ be
the closed interval $\left[  0,i\right]  $, subdivided into $i$ subintervals
of length $1$. Define three different sequences of retractions of $C_{i}$ onto
$C_{i-1}$.

\begin{itemize}
\item[i)] \label{Subexample: ray compactification with point}Let
$r_{i}:C_{i-1}\leftarrow C_{i}$ be the retraction that maps $\left[
i-1,i\right]  $ to its left endpoint.

\item[ii)] \label{Subexample: sin(1/x) curve}Let $r_{i}^{\prime}%
:C_{i-1}\leftarrow C_{i}$ be the retraction that maps $\left[  i-1,i\right]  $
linearly onto $\left[  i-2,i-1\right]  $ with $r_{i}\left(  i\right)  =i-1$; and

\item[iii)] \label{Subexample: bucket handle}let $r_{i}^{\prime\prime}%
:C_{i-1}\leftarrow C_{i}$ be the retraction that maps $\left[  i-1,i\right]  $
linearly onto $\left[  0,i-1\right]  $ with $r_{i}\left(  i\right)  =0$.
\end{itemize}

In all three cases, $X_{D}$ is homeomorphic to the standard ray $[0,\infty)$.
In case \ref{Subexample: ray compactification with point}, $X_{I}$ is also
homeomorphic to $[0,\infty)$, and $\overline{X_{I}}$ is the one-point
compactification $\left[  0,\infty\right]  $. In case
\ref{Subexample: sin(1/x) curve}, $X_{I}$ is again homeomorphic to
$[0,\infty)$, but $\overline{X_{I}}$ is a version of the \textquotedblleft
topologist's $\sin\left(  1/x\right)  $ curve\textquotedblright\ with $R$
corresponding to the limit arc.

In case \ref{Subexample: bucket handle}, $\overline{X_{I}}$ is homeomorphic to
a famous indecomposable continuum $B$, known as the \emph{bucket handle}\ or
the \emph{Brouwer-Janiszewski-Knaster continuum}. Here $R$ and $X_{I}$ are
both dense in $\overline{X_{I}}$, and the topology $\mathcal{T}_{I}$ on
$[0,\infty)$ is not the standard one. Open sets are most easily seen by taking
preimages of open subsets of $B$ under a rather exotic injection
$e:[0,\infty)\rightarrow B$ under which the sequence $\left\{  e\left(
n\right)  \right\}  _{n\in%
\mathbb{Z}
^{+}}$ converges to $e\left(  0\right)  $.\footnote{We are abusing notion
slightly by calling this map $e$.} As such, every neighborhood of $0$ in
$\mathcal{T}_{I}$ contains an unbounded sequence of open intervals. Other
points in $\left(  [0,\infty),\mathcal{T}_{I}\right)  $ have similar
neighborhoods. See Figure 2. One takeaway from this example is that, even when
$X_{D}$ is metrizable---or even a locally finite simplicial complex---the
space $X_{I}$ can be quite different. The example also illustrates the
significance of the insulation conditions; sequences
\ref{Subexample: ray compactification with point} and
\ref{Subexample: sin(1/x) curve} are fully insulated while sequence
\ref{Subexample: bucket handle} is not.
\end{example}

\begin{figure}[th]
\begin{center}
\includegraphics[scale=1.0]{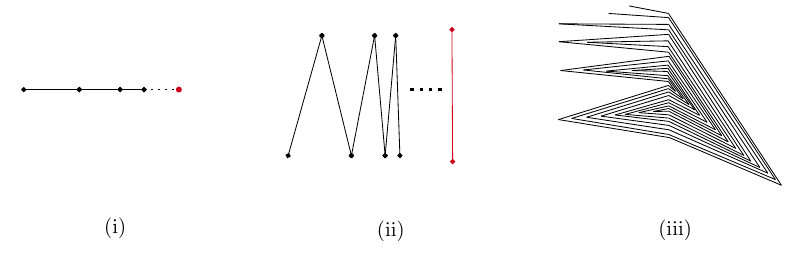} \label{Figure 2}
\end{center}
\caption{$\overline{X_{I}}=X_{I}\sqcup R$ in Example
\ref{Example: compactifications of a ray}}%
\label{Figure 2}%
\end{figure}

\begin{example}
\label{Example: infinite cube}For each integer $i\geq0$, let $I^{i}$ be the
standard $i$-cube
\[%
\begin{array}
[c]{ccc}%
\underbrace{\left[  0,1\right]  \times\cdots\left[  0,1\right]  } & \times &
0\times0\times\cdots\\
i\text{-times} &  &
\end{array}
\]
where the additional $0$-coordinates are added for notational convenience, and
so that $\left\{  I^{i}\right\}  _{i=0}^{\infty}$ forms a nested sequence. For
$i>0$, let $r_{i}:I^{i-1}\leftarrow I^{i}$ be the projection map. We may view
$X=\cup_{i=0}^{\infty}I^{i}$ as a subset $I_{0}^{\omega}$ of the countable
infinite product $I^{\omega}$; in particular, $I_{0}^{\omega}$ consists of all
sequences of elements of $\left[  0,1\right]  $ that end in a sequence of
$0$'s. An argument similar to the one alluded to in Example
\ref{Example: cone on a sequence 1} shows that, under the $\mathcal{T}_{D}$
(which is commonly used when $I_{0}^{\omega}$ is viewed as a CW complex),
$I_{0}^{\omega}$ is nonmetrizable. By contrast, $\mathcal{T}_{I}$ is the same
as the subspace topology induced by the product topology on $I^{\omega}$.
Additionally, $\overline{I_{0}^{\omega}}$ (as defined in Section
\ref{Subsection: Inverse and direct limits associated with X}) is equivalent
to taking the closure of $I_{0}^{\omega}$ in $I^{\omega}$. These observation
require some work since, by definition, $e\left(  X\right)  $ and
$\underleftarrow{\lim}\left\{  I^{i},r_{i}\right\}  $ are somewhat thin
subsets of $I^{0}\times I^{1}\times I^{2}\times\cdots$.

For readers interested in non-locally finite \emph{cube complexes},
$I_{0}^{\omega}$ is a beautifully simple example. In that context, it is
common to topologize $I_{0}^{\omega}$ using the path length metric $d$ induced
by the Euclidean metric on each $I^{i}$. Interestingly, that topology is
neither $\mathcal{T}_{D}$ nor $\mathcal{T}_{I}$. Being metric, it it is
clearly not $\mathcal{T}_{D}$. On the other hand, if $\mathbf{e}_{i}$ denotes
the point with all coordinates $0$ except for $1$ as its $\left(  i+1\right)
^{\text{st}}$ coordinate, then $\left\{  \mathbf{e}_{i}\right\}  $ converges
to $\left(  0,0,0,\cdots\right)  $ under $\mathcal{T}_{I}$, but has no limit
points under $d$ since $d\left(  \mathbf{e}_{i},\mathbf{e}_{j}\right)
=\sqrt{2}$ for all $i\neq j$.
\end{example}

\begin{example}
\label{Example: tangent disks}Let $C_{0}$ be the origin $\mathbf{0}$ in $%
\mathbb{R}
^{2}$, $C_{j}$ the closed circular disk of radius $j$ centered at $\left(
0,j\right)  $, and $r_{j}:C_{j-1}\leftarrow C_{j}$ the map taking $x$ to the
first point where the ray $\overrightarrow{x\mathbf{0}}$ intersects $C_{j-1}$.
Then $X_{I}$ is the open upper half-plane plus the point $\mathbf{0}$ with the
subspace topology, and $\overline{X_{I}}$ can be seen as the result of adding
an open interval $\left(  0,\pi\right)  $ to $X_{I}$, with the one-point for
each ray in $X_{I}$ emanating from $\mathbf{0}$. Another way to envision this
example is to embed the construction in the unit disk with $\left(
0,-1\right)  $ as the point of tangency. In this approach, $\overline{X_{I}}$
can be realized as the unit disk and the remainder $R$ as $\mathbb{S}%
^{1}-\{\left(  0,-1\right)  \}$.

In this example, each $C_{i}$ is stationary but not insulated. As a
consequence, Theorem \ref{Theorem: TFAE with T_I = T_D} cannot be applied, and
we end up with a space $X_{I}$ that is not locally compact and not an open
subset of $\overline{X_{I}}$. The reader is invited to compare $\mathcal{T}%
_{I}$ and $\mathcal{T}_{D}$.
\end{example}

\begin{example}
Let $C_{0}$ be origin and, for $j>0$, let $C_{j}$ be the convex hull of
\[
A_{j}=\left\{  \exp\left(  \frac{k\pi i}{2^{j-1}}\right)  \mid k=0,1,2,\cdots
,2^{j}-1\right\}
\]
in $%
\mathbb{R}
^{2}$, and $r_{j}:C_{j}\rightarrow C_{j-1}$ the radial projection of $C_{j}$
onto $C_{j-1}$. Then $X_{I}$ is the open unit disk together with the countable
dense set $\cup_{j\geq1}A_{j}$ of points on the unit circle (with the subspace
topology); and $\overline{X_{I}}$ is homeomorphic to the closed disk. As in
the previous example, $X_{I}$ is not locally compact and is not open in
$\overline{X_{I}}$.
\end{example}

\begin{example}
\label{Example: CAT(0) spaces}Let $\left(  Y,d\right)  $ be a proper CAT(0)
space and $y_{0}$ an arbitrary base point. Let $C_{0}=\left\{  y_{0}\right\}
$ and, for each $i>0$, $C_{i}=B_{d}\left[  y_{0},i\right]  $, the closed ball
of radius $i$. Let $r_{i}:C_{i-1}\leftarrow C_{i}$ be the geodesic retraction
taking $y\in C_{i}$ to the nearest point in $C_{i-1}$. As a set, $X=Y$, and
since each $C_{i}$ is insulated, $\mathcal{T}_{I}=\mathcal{T}_{D}$. With a
little effort, one sees that these topologies agree with original metric
topology and that $\overline{X_{I}}$ is homeomorphic to the compactification
$\overline{Y}:=Y\cup\partial_{\infty}Y$ obtained by adjoining the visual
sphere at infinity to $Y$ via the \textquotedblleft cone
topology\textquotedblright.

Note also that the only points of $C_{i}=B_{d}\left[  y_{0},i\right]  $ with
(possibly) non-singleton preimages in $C_{i+1}=B_{d}\left[  y_{0},i+1\right]
$ lie in the metric sphere $S_{d}[y_{0},i]$. It follows that
\[
R=\overline{X_{I}}-X_{I}=\underleftarrow{\lim}\left\{  C_{i},r_{i}\right\}
-e\left(  X\right)
\]
is equal to the inverse limit of the sequence
\[
\left\{  y_{0}\right\}  \overset{f_{1}}{\longleftarrow}S_{d}[y_{0}%
,1]\overset{f_{2}}{\longleftarrow}S_{d}[y_{0},2]\overset{f_{3}}{\longleftarrow
}\cdots
\]
where $f_{i}$ is the restriction of $r_{i}$. See \cite[\S II.8]{BrHa99} for
more details related to this construction. We will revisit to this example in
Section \ref{Subsection: Applications to CAT(0) spaces} and prove some
additional properties.
\end{example}

\begin{example}
\label{Example: Mapping telescope}Let
\begin{equation}
K_{0}\overset{f_{1}}{\longleftarrow}K_{1}\overset{f_{2}}{\longleftarrow}%
K_{2}\overset{f_{3}}{\longleftarrow}\cdots\label{inverse sequence of Ki}%
\end{equation}
be an inverse sequence of compact metric spaces and, for each integer $i>0$,
let
\[
\operatorname*{Map}\left(  f_{i}\right)  :=K_{i}\times\left[  i-1,i\right]
\sqcup K_{i-1}/\sim
\]
(the \emph{mapping cylinder}) be the quotient space, where\ $\sim$ is
generated by the rule $\left(  x,i-1\right)  \sim f_{i}\left(  x\right)  $. By
identifying $K_{i-1}$ with its image in $\operatorname*{Map}\left(
f_{i}\right)  $ and $K_{i}$ with the image of $K_{i}\times\left\{  i\right\}
$, we view $K_{i-1}$ and $K_{i}$ as subsets of $\operatorname*{Map}\left(
f_{i}\right)  $ and refer to them as the \emph{range end} and the \emph{domain
end }of $\operatorname*{Map}\left(  f_{i}\right)  $. There is a retraction
of\emph{ }$\operatorname*{Map}\left(  f_{i}\right)  $ onto $K_{i-1}$ (often
called the \emph{mapping cylinder collapse}) which sends $\left(  x,t\right)
$ to $f\left(  x\right)  $.

Let $C_{0}=K_{0}$ and for each $i>0$, let
\[
C_{i}=\operatorname*{Map}\left(  f_{1}\right)  \cup_{K_{1}}\operatorname*{Map}%
\left(  f_{2}\right)  \cup_{K_{2}}\cdots\cup_{K_{i-1}}\operatorname*{Map}%
\left(  f_{i}\right)
\]
By applying the mapping cylinder collapses, we get retractions $r_{i}%
:C_{i-1}\leftarrow C_{i}$ which can be assembled into an inverse sequence of
compact metrizable spaces%
\[
C_{0}\overset{r_{1}}{\longleftarrow}C_{1}\overset{r_{2}}{\longleftarrow}%
C_{2}\overset{r_{3}}{\longleftarrow}\cdots
\]
In this case
\[
X=\cup_{i=0}^{\infty}C_{i}=\operatorname*{Map}\left(  f_{1}\right)
\cup_{K_{1}}\operatorname*{Map}\left(  f_{2}\right)  \cup_{K_{2}}\cdots
\]
is called the \emph{inverse mapping telescope }of sequence
(\ref{inverse sequence of Ki}).

By reasoning similar to the previous example, $\mathcal{T}_{I}=\mathcal{T}%
_{D}$ and $R=\underleftarrow{\lim}\left\{  C_{i},r_{i}\right\}  -e\left(
X\right)  $ is the inverse limit of sequence (\ref{inverse sequence of Ki}).
We will return to this compactification in Section \ref{Section: Examples}.
\end{example}

\section{Expansions and collapses in the simplicial
category\label{Section: Expansions and collapses in the simplicial category}}

For our favorite applications of the above setting, we will need to place
conditions on the retraction maps in our inverse sequences. Necessity will
direct us to the notion of a \emph{compact collapse}---a generalization of a
\emph{simplicial collapse}. We will also need to all for infinitely many such
moves (or their inverses). Before proceeding to that level of generality, we
spend some time reviewing the classical theory and adjusting it to allow for
infinitely many moves.

\subsection{Simplicial complexes: preliminaries}

An \emph{abstract simplicial complex }$K$ is a pair $\left(  V,\mathcal{S}%
\right)  $ where $V$ is a set and $\mathcal{S}$ is a collection of nonempty
finite subsets of $V$ containing all singletons and satisfying the rule: if
$\sigma\in\mathcal{S}$ and $\varnothing\neq\tau\subseteq\sigma$, then $\tau
\in\mathcal{S}$. The elements of $V$ are called \emph{vertices} and the
elements of $\mathcal{S}$ the \emph{simplices} of $K$. (Typically, we do not
distinguish between a vertex $v$ and the simplex $\left\{  v\right\}  $.) If
$\tau\subseteq\sigma\in\mathcal{S}$, we call $\tau$ is a \emph{face} of
$\sigma$; if $\tau\subsetneq\sigma$, it is called a \emph{proper face}. If
$\sigma$ contains exactly $n+1$ vertices, it is called an $n$%
\emph{-dimensional simplex} (or simply an $n$\emph{-simplex}), and we
sometimes denote it by $\sigma^{n}$. A \emph{subcomplex} of $K$ is a
simplicial complex $L$, each of whose simplices is a simplex of $K$; this will
be denoted $L\leq K$. As a key example, if $\sigma$ is a simplex of $K$ then
$\partial\sigma\leq K$ is made up of all proper faces of $\sigma$. A
simplicial complex is \emph{finite }if its vertex set is finite; it is
\emph{locally finite} if each vertex belongs to only finitely many simplices.

We define the \emph{geometric realization }$\left\vert K\right\vert $ of
$K=\left(  V,\mathcal{S}\right)  $ (also called the \emph{polyhedron
corresponding to }$K$) to be a subset of the real vector space $W_{V}$ with
basis $V$. By endowing $W_{V}$ with the standard inner product metric $d_{V}$,
each subset $F\subseteq V$ containing $k<\infty$ elements, is the basis for a
subspace $W_{F}$ that is isometric to $%
\mathbb{R}
^{k}$ with the Euclidean metric. For $\sigma=\left\{  v_{0},\cdots
,v_{n}\right\}  \in\mathcal{S}$, the \emph{geometric realization} $\left\vert
\sigma\right\vert $ is the convex hull of $\left\{  v_{0},\cdots
,v_{n}\right\}  $ in $W_{\sigma}$ with the subspace metric. Now define
$\left\vert K\right\vert \subseteq W_{v}$ to be the union of the geometric
realizations of its simplices.

Note that%
\[
\left\vert \sigma\right\vert =\left\{  \left.  \sum\nolimits_{i=0}^{n}%
x_{i}v_{i}\right\vert 0\leq x_{i}\leq1\text{ and }\sum\nolimits_{i=0}^{n}%
x_{i}=1\right\}
\]
If $\mathbf{x}=\sum\nolimits_{i=0}^{n}x_{i}v_{i}\in\left\vert \sigma
\right\vert $, the $x_{i}$ are called the \emph{barycentric coordinates} of
$\mathbf{x}$. The \emph{barycenter} of $\left\vert \sigma\right\vert $ is the
point $\widehat{\sigma}=\sum\nolimits_{i=0}^{n}\frac{1}{n+1}v_{i}$. We will
refer to any space isometric to $\left\vert \sigma\right\vert $ as a
\emph{standard Euclidean }$n$\emph{-simplex}, in which case the distance
between any pair of distinct vertices is $\sqrt{2}$.\footnote{Some authors
rescale the metric so that the distance between pairs of vertices in a
standard Euclidean $n$-simplex is $1$. For our purposes, that is unnecessary.}

There are multiple ways to topologize $\left\vert K\right\vert $. One can give
$\left\vert K\right\vert $ the initial topology with respect to the individual
simplices; in other words, $U\subseteq\left\vert K\right\vert $ is open if and
only if $U\cap\left\vert \sigma\right\vert $ is open for every $\sigma
\in\mathcal{S}$. This is often called the \emph{weak topology} on $\left\vert
K\right\vert $, so we denote it $\mathcal{T}_{W}$. Alternatively, one can use
the inner product metric $d_{V}$ on $W_{V}$ to induce the subspace topology on
$\left\vert K\right\vert $. We call this the \emph{barycentric topology }and
denote it $\mathcal{T}_{B}$. It is a classical fact that $\mathcal{T}%
_{W}=\mathcal{T}_{B}$ if and only if $K$ is locally finite; otherwise,
$\mathcal{T}_{W}\supsetneq\mathcal{T}_{B}$ and the identity map
$\operatorname*{id}_{\left\vert K\right\vert }:\left(  \left\vert K\right\vert
,\mathcal{T}_{W}\right)  \rightarrow\left(  \left\vert K\right\vert
,\mathcal{T}_{B}\right)  $ is a homotopy equivalence. See \cite{Dow52}.

\begin{remark}
\label{Remark: path length metric vs barycentric metric}For the purposes of
metric geometry and geometric group theory, the \emph{path length metric on
}$\left\vert K\right\vert $ (also called the \emph{intrinsic metric} in
\cite[Chap.I.7]{BrHa99}) is often more useful than the barycentic metric. For
example, the latter is always bounded and thus ignores the large-scale
geometry of $\left\vert K\right\vert $. Nevertheless, for each $\mathbf{x}%
\in\left\vert K\right\vert $, there is an $\varepsilon>0$ such that balls
centered at $\mathbf{x}$ and having radius $<\varepsilon$ are identical under
the two metrics. As such they generate the same topologies.
\end{remark}

Abstract simplicial complexes $K_{1}$ and $K_{2}$ are \emph{isomorphic},
denoted $K_{1}\cong K_{2}$, if there exists a bijection $f:V_{1}\rightarrow
V_{2}$ between vertex sets that induces a bijection between the sets of
simplices. Clearly $f$ induces a corresponding homeomorphism $\left\vert
f\right\vert :\left\vert K_{1}\right\vert \rightarrow\left\vert K_{2}%
\right\vert $ provided $\left\vert K_{1}\right\vert $ and $\left\vert
K_{2}\right\vert $ are topologized in the same manner. Polyhedra $\left\vert
L_{1}\right\vert $ and $\left\vert L_{2}\right\vert $ are called
\emph{piecewise-linearly (PL) homeomorphic }if there exist isomorphic
subdivisions $L_{1}^{\prime}$ and $L_{2}^{\prime}$ of $L_{1}$ and $L_{2}$.

A \emph{triangulation} of a topological space $X$ is a homeomorphism
$h:\left\vert K\right\vert \rightarrow X$ for some simplicial complex $K$. In
that case, we say $X$ is \emph{triangulated} by $K$.

\subsection{Simplicial expansions and
collapses\label{Subsection: simplicial expansions and collapses}}

An \emph{elementary combinatorial collapse} $K\overset{e}{\searrow}L$ of $K$
onto a subcomplex $L$ can be performed when $K$ contains an $n$-simplex
$\sigma^{n}$ that has a \emph{free }face $\tau^{n-1}$ (meaning that
$\tau^{n-1}$ is a face of no other simplices of $K$). The subcomplex $L$ is
obtained by deleting $\sigma^{n}$ and $\tau^{n-1}$ from $K$. This
combinatorial move can be realized topologically by a particularly nice type
of deformation retraction of $\left\vert K\right\vert $ onto $\left\vert
L\right\vert $. Due to its importance in this paper, we will carefully
describe one such deformation retraction in the following paragraphs. Before
that, we note that the reverse of an elementary combinatorial collapse is
called an \emph{elementary combinatorial expansion} and denoted
$L\overset{e}{\nearrow}K$.

To define the deformation retraction alluded to above, label the vertices of
$\sigma^{n}$ by $\left\{  v_{0},\cdots,v_{n}\right\}  $ such that $\left\{
v_{0},\cdots,v_{n-1}\right\}  $ are the vertices of $\tau^{n-1}$, and (abusing
notation slightly) choose a standard Euclidean realization of $\left\vert
\sigma^{n}\right\vert $ in $%
\mathbb{R}
^{n}$ with $\left\vert \tau^{n-1}\right\vert \subseteq%
\mathbb{R}
^{n-1}\times\left\{  0\right\}  $ having barycenter at the origin and so that
$v_{n}=\left(  0,\cdots,0,\sqrt{\frac{n+1}{n}}\right)  $. Let $L^{\prime}$ be
the simplicial complex $\partial\sigma^{n}-\tau^{n-1}$. For each
$\mathbf{x\in}\left\vert \sigma^{n}\right\vert $, define $r\left(
\mathbf{x}\right)  $ be the unique point where the ray $\overrightarrow{-v_{n}%
\mathbf{x}}$ intersects $\left\vert L^{\prime}\right\vert $. Then define
$J:\left\vert \sigma^{n}\right\vert \times\lbrack0,1]\longrightarrow\left\vert
\sigma^{n}\right\vert $ to be the map for which $\left.  J\right\vert
_{\mathbf{x\times}\left[  0,1\right]  }$ is the constant speed linear path
from $\mathbf{x}$ to $r\left(  \mathbf{x}\right)  $. In particular, $J$ is
fixed on $\left\vert L^{\prime}\right\vert $, $J_{0}=\operatorname*{id}%
_{\left\vert \sigma^{n}\right\vert }$, and $J_{1}=r$ is a retraction of
$\left\vert \sigma^{n}\right\vert $ onto $\left\vert L^{\prime}\right\vert $.
Now define $H:\left\vert K\right\vert \times\lbrack0,1]\longrightarrow
\left\vert L\right\vert $ to be $J$ on $\left\vert \sigma^{n}\right\vert
\times\lbrack0,1]$ while keeping all points of $\left\vert L\right\vert $
fixed throughout. By a mild abuse of notation, we also let $r$ denote the map
$H_{1}:\left\vert K\right\vert \rightarrow\left\vert L\right\vert $ and call
it an \emph{elementary simplicial collapse}. The corresponding inclusion map
$s:\left\vert L\right\vert \hookrightarrow\left\vert K\right\vert $ is called
an \emph{elementary simplicial expansion}. It is easy to see that $H$ is
continuous under all of the above-mentioned topologies on $\left\vert
K\right\vert $.

Now define a \emph{combinatorial collapse} of $K$ to a subcomplex $L_{0}$ to
be a finite sequence of elementary combinatorial simplicial collapses
$K=L_{n}\overset{e}{\searrow}L_{n-1}\overset{e}{\searrow}\cdots
\overset{e}{\searrow}L_{0}$, a process denoted loosely by $K\searrow L_{0}$.
The concatenation of the corresponding deformation retractions, gives a
deformation retraction of $\left\vert K\right\vert $ onto $\left\vert
L_{0}\right\vert $ whose end result $r:\left\vert K\right\vert \rightarrow
\left\vert L_{0}\right\vert $ is a retraction that we call a \emph{simplicial
collapse}. Note that $r$ depends on the specific sequence of elementary
collapses used to take $K$ onto $L_{0}$ but, regardless of the sequence
chosen, it is a homotopy equivalence, with homotopy inverse $\left\vert
L_{0}\right\vert \hookrightarrow\left\vert K\right\vert $. Under these
circumstances we also write $L_{0}\overset{e}{\nearrow}L_{1}%
\overset{e}{\nearrow}\cdots\overset{e}{\nearrow}K$, or simply $L_{0}\nearrow
K$, and say that $L_{0}$ \emph{expands} \emph{combinatorially} to $K$ and that
$\left\vert L_{0}\right\vert $ \emph{expands simplicially} to $\left\vert
K\right\vert $.

\begin{example}
\label{Example: simplicially collapsing a product}It is a standard (and
useful) fact that, for every finite simplicial complex $K$, there exists a
triangulation of $\left\vert K\right\vert \times\left[  0,1\right]  $ by a
simplicial complex $L$ with subcomplexes $K\times\left\{  i\right\}  $
corresponding to $\left\vert K\right\vert \times\left\{  i\right\}  $ for
$i=0,1$, and such that $\left\vert L\right\vert $ collapses simplicially to
$\left\vert K\times\left\{  i\right\}  \right\vert $ for $i=0,1$.
\end{example}

A simplicial complex $K$ is \emph{combinatorially collapsible} if it contains
a vertex $v$ such that $K\searrow\left\{  v\right\}  $ or, equivalently, there
exists a sequence of elementary combinatorial expansions $\left\{  v\right\}
=L_{0}\overset{e}{\nearrow}L_{1}\overset{e}{\nearrow}\cdots
\overset{e}{\nearrow}L_{n}=K$ where each $L_{i}\leq K$. In this case, we say
that $\left\vert K\right\vert $ is \emph{simplicially collapsible}.

\begin{remark}
The geometric realization of a collapsible simplicial complex is clearly
contractible. In fact, collapsibility is often viewed as the combinatorial
analog of contractibility. Nevertheless, there are well known examples, such
as the \emph{topologist's dunce hat} and \emph{Bing's house with two rooms},
that are contractible but not collapsible. Also clear, from the above
definition, is that a collapsible simplicial complex is necessarily finite. We
will soon relax key definitions to eliminate that requirement.
\end{remark}

\begin{remark}
The subject of \emph{simple homotopy theory} defines a pair of finite
simplicial complexes $K$ and $L$ to be simple homotopy equivalent if,
beginning with $K$, there exists a finite sequence of elementary expansions
and collapses ending in $L$. As an example, the dunce hat and Bing's house
with two rooms---despite the fact that they are not collapsible---are simple
homotopy equivalent to a vertex. By contrast, there exist homotopy equivalent
simplicial complexes that are not simple homotopy equivalent. Siebenmann
\cite{Sie70a} has developed a simple homotopy theory for infinite complexes,
but it differs substantially from what we will do here. Simple homotopy theory
does not play a significant role in this paper.
\end{remark}

\subsection{Infinite simplicial expansions and
collapses\label{Subsection: Infinite simplicial expansions and collapses}}

We now extend the classical notions of collapses and expansions to allow for
infinitely many moves. We say that a simplicial complex $K$ \emph{collapses
combinatorially to} $L$, or equivalently, that $L$ \emph{expands
combinatorially to }$K$, if there exists a (possibly infinite) sequence of
elementary combinatorial simplicial expansions $L=L_{0}\overset{e}{\swarrow
}L_{1}\overset{e}{\swarrow}L_{2}\overset{e}{\swarrow}\cdots$ through
subcomplexes of $K$, such that $\cup_{i=0}^{\infty}L_{i}=K$. In this case, we
write $K\searrow L$ or $L\nearrow K$. When extra clarity is desired (and
infinitely many moves are used) we will write $K\overset{\infty}{\searrow}L$
or $L\overset{\infty}{\nearrow}K$. We say that $K$ is \emph{combinatorially
collapsible} if $\left\{  v\right\}  \nearrow K$, with the option of
specifying $\left\{  v\right\}  \overset{\infty}{\nearrow}K$, when infinitely
many expansions are required.

Although these new definitions are straightforward and natural, some oddities
result. For example, a simplicial complex can be collapsible without having a
free face. For example, the standard triangulation of $%
\mathbb{R}
$ is collapsible. Despite such oddities, the generalized definitions are
useful. To begin, notice that the assumption $L\nearrow K$ allows one to
construct a direct sequence of geometric realizations and inclusion maps%
\begin{equation}
\left\vert L\right\vert =\left\vert L_{0}\right\vert \overset{s_{1}%
}{\hookrightarrow}\left\vert L_{1}\right\vert \overset{s_{2}}{\hookrightarrow
}\left\vert L_{2}\right\vert \overset{s_{3}}{\hookrightarrow}\cdots
\label{direct sequence of simplicial complexes}%
\end{equation}
and an inverse sequence of (deformation) retractions.
\begin{equation}
\left\vert L\right\vert =\left\vert L_{0}\right\vert \overset{r_{1}%
}{\longleftarrow}\left\vert L_{1}\right\vert \overset{r_{2}}{\longleftarrow
}\left\vert L_{2}\right\vert \overset{r_{3}}{\longleftarrow}\cdots
\label{inverse sequence of simplicial complexes}%
\end{equation}
We say that $\left\vert K\right\vert $ \emph{collapses simplicially to}
$\left\vert L\right\vert $ and that $\left\vert L\right\vert $ \emph{expands
simplicially to }$\left\vert K\right\vert $. In this context, we always assume
that the retractions used in sequence
(\ref{inverse sequence of simplicial complexes}) are the simplicial collapses
defined earlier. When $\left\vert K\right\vert $ is endowed with the
barycentric or weak topology, a straightforward application of Whitehead's
Theorem \cite{Hat02} ensures that, $\left\vert L\right\vert \hookrightarrow
\left\vert K\right\vert $ is a homotopy equivalence---as it was in the finite
case. If we opt for the topology $\mathcal{T}_{I}$ on $\left\vert K\right\vert
$, which is often the case in this paper, proving the inclusion is a homotopy
equivalence is more subtle; that will be done in Section
\ref{Subsection: Homotopy negligible compactifications}. For now we observe
that, in a case of primary interest, everything works out nicely.

\begin{proposition}
\label{Proposition: equivalence of the topologies}Given the above setup, if
$K$ is locally finite, then the topologies $\mathcal{T}_{I}$, $\mathcal{T}%
_{D}$, $\mathcal{T}_{W}$ and $\mathcal{T}_{B}$ on $\left\vert K\right\vert $
are the same.
\end{proposition}

\begin{proof}
By definition and \cite{Dow52}, we know $\mathcal{T}_{D}=\mathcal{T}%
_{W}=\mathcal{T}_{B}$, so it suffices to show that $\mathcal{T}_{D}%
=\mathcal{T}_{I}$. By Theorem \ref{Theorem: TFAE with T_I = T_D}, we need only
show that every $x\in\left\vert K\right\vert $ is insulated. Toward that end,
choose $j$ sufficiently large that every simplex of $\left\vert K\right\vert $
that contains $x$ is contained in $\left\vert L_{j}\right\vert $ and let $U$
be the open star neighborhood of $x$ in $\left\vert L_{j}\right\vert $ (and
hence in $\left\vert K\right\vert $). We claim that $r_{j\infty}^{-1}\left(
U\right)  =U$. If not, let $k>j$ be the smallest integer such that
$r_{jk}^{-1}\left(  U\right)  $ contains a point not in $U$. Then $\left\vert
L_{k}\right\vert =\left\vert L_{k-1}\right\vert \cup\tau$ and, by definition,
$r_{k}$ maps $\tau$ onto a union of its proper faces. By our assumption, one
of those faces lies in $U$, and hence contains $x$. But then, $\tau$ is a
simplex of $K$ that contains $x$ and hence lies in $U$, a contradiction.
\end{proof}

\section{Compact expansions and
collapses\label{Section: Topological expansions and collapses}}

We are ready to expand the notions of expansion and collapse beyond the
simplicial category. That, in itself, is not new. There are well-known
extensions to:

\begin{itemize}
\item polyhedral complexes, where collapses are performed on convex polyhedra;

\item $n$-manifolds, where topological $n$-balls are collapsed onto $\left(
n-1\right)  $-balls in their boundary (a process known as \emph{shelling}); and

\item CW complexes, where $k$-cells of various dimensions are collapsed onto
potentially crumpled $\left(  k-1\right)  $-cells in their attaching spheres.
\end{itemize}

\noindent The CW version is often viewed as the most general context for
defining expansions and collapses. We will argue otherwise, by formulating a
version that includes all of the above as special cases. It makes room for new
examples and allows some well known constructions to be viewed as special
cases of a broader theory.

\subsection{Track-faithful deformation retractions}

Let $H:C\times\left[  0,1\right]  \rightarrow C$ be a self-homotopy of a space
$C$ and, for each $t\in\left[  0,1\right]  $, let $H_{t}:C\rightarrow C$ be
defined by $H_{t}\left(  x\right)  =H\left(  x,t\right)  $. For our purposes,
we are only interested cases where $H_{0}=\operatorname*{id}_{C}$. We call $H$
a \emph{deformation retraction} of $C$ onto $A$ if $A=H_{1}\left(  C\right)  $
and $\left.  H_{t}\right\vert _{A}=\operatorname*{id}_{A}$ for all
$t$.\footnote{Some authors call this a \emph{strong} deformation
retraction.}In that case, we call the retraction of $C$ onto $A$, obtained by
restricting the range of $H_{1}$, the \emph{culmination }of $H$.

The \emph{track} of a point $x\in C$ under $H$ is the path $\lambda
_{x}:\left[  0,1\right]  \rightarrow C$ defined by $\lambda_{x}\left(
t\right)  =H_{t}\left(  x\right)  $. We sometimes abuse terminology by
conflating $\lambda_{x}$ with its image. Call $H$ \emph{track-faithful }if,
for every $x\in C$ and $t\in\left[  0,1\right]  $, $H_{1}\left(  H_{t}\left(
x\right)  \right)  =H_{1}\left(  x\right)  $. This condition is equivalent to
requiring that, if two tracks intersect, then they end at the same point. We
are particularly interested in track-faithful deformation retractions.

\begin{example}
\label{Example: product collapse}If $C=X\times\left[  0,1\right]  $ and
$H_{t}\left(  x,s\right)  =\left(  x,\left(  1-t\right)  s\right)  $ then $H$
is a track-faithful deformation retraction of $C$ onto $X\times\left\{
0\right\}  $.
\end{example}

\begin{example}
\label{Example: mapping cylinder collapse}For any map $f:X\rightarrow A$, the
deformation retraction described in Example \ref{Example: product collapse}
induces a track-faithful deformation retraction $\operatorname*{Map}\left(
f\right)  $ onto $A$. By contrast, whenever $f$ is a homotopy equivalence,
$\operatorname*{Map}\left(  f\right)  $ deformation retracts onto $X$, but in
some cases there is no track-faithful deformation retraction. An example with
$X=A=S^{1}$ can be obtained by comparing \cite{Gui01} with Theorem
\ref{Theorem: homotopy negligible compactifications} in this paper.
\end{example}

\begin{example}
\label{Example: elementary simplicial collapses}Given a top-dimensional face
$\tau^{n-1}$ of an $n$-simplex $\sigma^{n}$ and the subcomplex $L^{\prime
}=\partial\sigma^{n}-\tau^{n-1}$, the deformation retractions of $\left\vert
\sigma^{n}\right\vert $ onto $\left\vert L^{\prime}\right\vert $ and
$\left\vert K\right\vert $ onto $\left\vert L\right\vert $, described in
Section \ref{Section: Expansions and collapses in the simplicial category} and
used to define an elementary simplicial collapse, are track-faithful.
\end{example}

\begin{example}
\label{Example: elementary cubical collapse}Let $\square^{n}=\left[
-1,1\right]  ^{n}$, $B^{n-1}\ =\square^{n-1}\times\left\{  1\right\}  $, and
$J^{n-1}=\partial\square^{n}-\operatorname*{int}B^{n-1}$. Let $\mathbf{p}%
=\left(  0,\cdots,0,2\right)  $ and for each $\mathbf{x}\in\square^{n}$, let
$\mathbf{q}_{x}$ be the unique point $\overrightarrow{\mathbf{px}}\cap J$.
This gives a continuously varying family of segments $\overline{\mathbf{xq}%
}_{\mathbf{x}}\subseteq\square^{n}$ (some of them degenerate) connecting
points of $\square^{n}$ to points of $J^{n-1}$. By choosing constant speed
parametrization $\lambda_{\mathbf{x}}:\left[  0,1\right]  \rightarrow
\overline{\mathbf{xq}}_{\mathbf{x}}$ of these segments we obtain a
track-faithful deformation retraction of $\square^{n}$ onto $J^{n-1}$. We will
refer to culmination as an \emph{elementary cubical collapse}.
\end{example}

\begin{example}
\label{Example: contractions that are track-faithful}A contraction of a space
$C$ to a point $p$ is track-faithful if and only if $p$ is fixed by the contraction.
\end{example}

\begin{example}
The retractions $r,r^{\prime}:\left[  -1,1\right]  \rightarrow\left[
0,1\right]  $ defined by $r\left(  t\right)  =\left\vert t\right\vert $ and
\[
r^{\prime}\left(  t\right)  =\left\{
\begin{tabular}
[c]{cc}%
$0$ & if $t\leq0$\\
$t$ & if $t\geq0$%
\end{tabular}
\ \ \ \ \ \right.
\]
can each be realized as the culmination of a deformation retraction. But,
whereas $r^{\prime}$ is the culmination of a track-faithful deformation
retraction, $r$ is not.
\end{example}

\begin{lemma}
Let $H_{t}$ be a track-faithful deformation retraction of $C$ onto $A$
culminating in $r:C\rightarrow A$. Then

\begin{enumerate}
\item for each $x\in A$, $r^{-1}\left(  x\right)  $ is contractible, and

\item if $x\in\operatorname*{Int}A$, then $r^{-1}\left(  x\right)  =\left\{
x\right\}  $.
\end{enumerate}
\end{lemma}

\begin{proof}
For the first assertion, note that if $y\in r^{-1}\left(  x\right)  $, then
$\lambda_{y}\subseteq r^{-1}\left(  x\right)  $. Therefore, the restriction of
$H$ to $r^{-1}\left(  x\right)  \times\left[  0,1\right]  $ is a contraction
of $r^{-1}\left(  x\right)  $. For the second assertion, let $x\in
\operatorname*{Int}A$ and suppose there exists $y\in r^{-1}\left(  x\right)  $
such that $y\neq x$. Then there exists $t_{0}\in\left[  0,1\right]  $ such
that $\lambda_{x}\left(  t_{0}\right)  \in A$ and $\lambda_{x}\left(
t_{0}\right)  \neq x$. Let $y_{0}=\lambda_{x}\left(  t_{0}\right)  $ and note
that $\lambda_{y_{0}}=\left\{  y_{0}\right\}  $. Our hypothesis ensures that
$r\left(  y_{0}\right)  =H_{1}\left(  y_{0}\right)  =H_{1}\left(  H_{t_{0}%
}\left(  y\right)  \right)  =x$, but since $y_{0}\in A$, $r\left(
y_{0}\right)  =y_{0}$.
\end{proof}

\begin{lemma}
If $H:C\times\left[  0,1\right]  \rightarrow C$ is a track faithful
deformation retraction of $C$ onto $A$ and $J:A\times\left[  0,1\right]
\rightarrow A$ is a track faithful deformation retraction of $A$ onto $B$,
then the concatenation of $H$ and $J$ is a track faithful deformation
retraction of $C$ onto $B$ with $\left(  H\ast J\right)  _{1}=J_{1}\circ
H_{1}$.
\end{lemma}

\begin{proof}
By \textquotedblleft concatenation of $H$ and $J$\textquotedblright, we mean
the homotopy $H\ast J:C\times\left[  0,1\right]  \rightarrow C$ defined by
\[
(H\ast J)\left(  x,t\right)  =\left\{
\begin{tabular}
[c]{cc}%
$H\left(  x,2t\right)  $ & if $0\leq t\leq\frac{1}{2}$\\
$J(x,2t-1)$ & if $\frac{1}{2}\leq t\leq1$%
\end{tabular}
\ \right.
\]
Verification of the desired properties is straightforward.
\end{proof}

\begin{remark}
There are other reasonable candidates for the definition of track-faithful
homotopy. For example, one could require that the track of every point lying
on $\lambda_{x}$ be contained in $\lambda_{x}$. We have chosen the weakest
condition that is required in our proofs.
\end{remark}

\subsection{Compact expansions and collapses}

Given a metric space $C$, a map $\rho:C\rightarrow A$ is an \emph{compact
collapse} if

\begin{itemize}
\item[\textbf{a)}] \label{Definition: elementary topological collapse a)}%
$\rho$ is the culmination of track-faithful deformation retraction, and

\item[\textbf{b)}] \label{Definition: elementary topological collapse b)}$C-A$
has compact closure in $C$.
\end{itemize}

\noindent In this case, we write $C\overset{\rho}{\searrow}A$ or
$A\overset{\rho}{\swarrow}C$; we call the inclusion map $\theta
:A\hookrightarrow C$ an \emph{compact expansion} and write $A\overset{\theta
}{\nearrow}C$ or $C\overset{\theta}{\nwarrow}A$.

\begin{notation}
As an aid to the reader, we will use the letter \textquotedblleft
rho\textquotedblright\ \emph{only} to denote compact collapses and
\textquotedblleft theta\textquotedblright\ only to denote compact expansions.
\end{notation}

\begin{example}
\label{Example: examples of compact collapses}Examples
\ref{Example: product collapse}, \ref{Example: mapping cylinder collapse} and
\ref{Example: contractions that are track-faithful} give compact collapses
when the corresponding spaces are compact. The elementary simplicial and
cubical collapses in Example \ref{Example: elementary simplicial collapses}
and \ref{Example: elementary cubical collapse} are also compact collapses.
\end{example}

\begin{lemma}
\label{Lemma: collapse is a topological property}If $C\overset{\rho}{\searrow
}A$ and the pair $\left(  C,A\right)  $ is homeomorphic to $\left(  C^{\prime
},A^{\prime}\right)  $, then $C^{\prime}\overset{\rho^{\prime}}{\searrow
}A^{\prime}$ for an appropriately chosen retraction $\rho^{\prime}$.
\end{lemma}

\begin{corollary}
\label{Corollary: ball collapse}Let $B^{n}$ be a topological $n$-ball and
$B^{n-1}$ be an $\left(  n-1\right)  $ ball tamely embedded in $\partial
B^{n}$. Then $B^{n}\overset{\rho}{\searrow}B^{n-1}$.
\end{corollary}

\begin{proof}
By standard methods in geometric topology $\left(  B^{n},B^{n-1}\right)
\approx\left(  \square^{n},J^{n-1}\right)  $. So the claim follows from Lemma
\ref{Lemma: collapse is a topological property} and Example
\ref{Example: elementary cubical collapse}.
\end{proof}

\begin{lemma}
\label{Lemma: extending a topological collapse}Suppose $X$ can be expressed as
a union $X=C\cup Y$ of closed subspaces; $A=C\cap Y$; and $C\overset{\rho
}{\searrow}A$. Then $X\overset{\rho}{\searrow}Y$. More generally, if
$C\overset{\rho}{\searrow}A$ and $f:A\rightarrow Y$ is continuous, then
$Y\cup_{f}C\overset{\rho}{\searrow}Y$.
\end{lemma}

\begin{proof}
The deformation retraction needed for the initial statement follows from the
Pasting Lemma. For the second statement, note that the track-faithful
deformation retraction of $C$ onto $A$ induces a track-faithful deformation
retraction of the induced quotient space $C/\sim$ onto $f\left(  A\right)
=A/\sim$. Since points in $f\left(  A\right)  $ have constant tracks, this
deformation retraction extends via the identity homotopy over $Y$.
\end{proof}

\begin{corollary}
\label{Corollary: CW collapse}Every elementary CW collapse (in the sense of
\cite[p.14]{Coh73}) of a CW complex $K$ to a subcomplex $L$ can be realized as
a compact collapse $K\overset{\rho}{\searrow}L$.
\end{corollary}

\begin{proof}
Apply Corollary \ref{Corollary: ball collapse} and Lemma
\ref{Lemma: extending a topological collapse} to the definition.
\end{proof}

The reader may have wondered why we did not include the word \textquotedblleft
elementary\textquotedblright\ in our definition of compact collapse. The
reason is that, unlike in the more prescriptive simplicial and CW versions,
any finite sequence of compact collapses can be combined into a single compact collapse.

\begin{lemma}
\label{Lemma: combining finitely many compact collapses}Suppose
\[
X_{n}\overset{\rho_{n}}{\searrow}X_{n-1}\overset{\rho_{n-1}}{\searrow}%
\cdots\overset{\rho_{1`}}{\searrow}X_{0}%
\]
is a finite sequence of compact collapses. Then $X_{n}\overset{\rho}{\searrow
}X_{0}$, where $\rho=\rho_{1}\circ\cdots\circ\rho_{n}$.
\end{lemma}

\begin{remark}
\label{Remark: on contractible spaces}Note that every base point preserving
contraction of a compactum $X$ is a compact collapse. If we were to allow
\textquotedblleft noncompact collapses\textquotedblright, then all reasonably
nice contractible spaces would be collapsible in a single move. That extreme
level of flexibility is unlikely to yield a useful theory. By contrast, we
will show that a theory built upon compact expansions and collapses can be
quite useful.
\end{remark}

\subsection{Infinite compact expansions and
collapses\label{Subsection: Infinite compact expansions and collapses}}

We are ready for our most general treatment of expansions and collapses---one
based on compact expansions and collapses and allowing infinitely many moves.

We say that a metric space $X$ \emph{collapses compactly }to a subspace $L$,
or equivalently, that $L$ \emph{expands compactly to }$X$, if there exists a
(possibly infinite) sequence of compact expansions $L=L_{0}%
\overset{e}{\swarrow}L_{1}\overset{e}{\swarrow}L_{2}\overset{e}{\swarrow
}\cdots$ through subsets of $X$, such that $\cup_{i=0}^{\infty}L_{i}=X$. In
this case, we write $X\searrow L$ or $L\nearrow X$. For extra clarity we
sometimes write $X\overset{\infty}{\searrow}L$ or $L\overset{\infty}{\nearrow
}X$. We say that $X$ is \emph{compactly collapsible} if $\left\{  p\right\}
\nearrow X$, for some $p\in X$. Note that, although we can combine finite
pieces of an infinite sequence of compact expansions into single compact
expansion, an infinite sequence will remain infinite. 

\begin{remark}
It is worth commenting further on the work in \cite{Fer80} mentioned in the
introduction. Ferry's motivation for studying compact collapses (which he
called \emph{contractible retractions}) was an attempt to develop a
\textquotedblleft simple homotopy theory for compact metric
spaces\textquotedblright. In analogy with the traditional theory, a pair of
compact metric spaces would be \emph{simple homotopy equivalent}\ if one could
be converted to the other by a finite sequence of compact expansions and
collapses. That idea was undone by the strength of Ferry's own work. He proved
that two compacta satisfy that version of simple homotopy equivalent, if and
only if they are homotopy equivalent---a nice theorem, but one which nullified
the utility of a new definition. Our work resurrects a portion that program,
but with different goals. By utilizing \emph{infinite} sequences of compact
expansions and collapses, we will obtain some new results. We are particularly
interested in $\mathcal{Z}$-compactifications. For that purposes, Ferry's
definition is a perfect fit.
\end{remark}

\section{Compactifications\label{Section: Compactifications}}

A primary goal of this paper is to develop general methods for obtaining nice
compactifications. We are especially interested in $\mathcal{Z}$%
-compactifications of spaces arising in geometric group theory and geometric
and algebraic topology; but much the machinery can be applied more widely to
obtain various interesting compactifications. To allow for the broadest set of
applications, we begin with a flexible set of definitions. The reader is
warned that the terminology surrounding this topic---including key
definitions, such as $\mathcal{Z}$-set---is not consistent across the
literature. Rather than trying to blend with everything written previously, we
will choose definitions that works for our purposes and causes as little
confusion as possible. In many cases, conflicts in terminology vanish in the
presence of certain \textquotedblleft niceness\textquotedblright\ assumptions
on the spaces to which it applies. For a thorough treatment of the various
definitions, with proofs of equivalences under appropriate conditions, the
reader is directed to \cite{Anc23}.

For this paper, there is no loss in restricting the following definitions to
cases where the ambient space $Y$ is compact and metrizable. In that case, we
can let $d$ be any metric inducing the correct topology on $Y$. For open cover
$\mathcal{U}$ of $Y$, maps $f,g:Y\rightarrow Y$ are $\mathcal{U}$\emph{-close
}if, for every $x\in Y$ there exists a $U\in\mathcal{U}$ such that $f\left(
x\right)  $ and $g\left(  x\right)  $ are contained in $U$. These maps are
called $\mathcal{U}$\emph{-homotopic} if there exists a homotopy $H_{t}$
between $f$ and $g$, each of whose tracks is contained in an element of
$\mathcal{U}$. In that case, $H_{t}$ is called a $\mathcal{U}$\emph{-homotopy.
}Clearly $\mathcal{U}$-homotopic maps are $\mathcal{U}$-close.

A subset $A$ of $Y$ is called \emph{mapping negligible} if, for every open
cover $\mathcal{U}$ of $Y$, there is a map $f:Y\rightarrow Y$ that is
$\mathcal{U}$-close to $\operatorname*{id}_{Y}$ and for which $f\left(
Y\right)  \subseteq Y-A$. We say that $A$ is \emph{homotopy negligible} if
there exists a homotopy $H:Y\times\left[  0,1\right]  \rightarrow Y$ such that
$H_{0}=\operatorname*{id}_{Y}$ and $H_{t}\left(  Y\right)  \subseteq Y-A$ for
all $t>0$. For an arbitrary open cover $\mathcal{U}$ of $Y$, one can
arrange---by restricting $H$ to a short interval $\left[  0,\delta\right]  $,
then reparametrizing---that this is accomplished by a $\mathcal{U}$-homotopy.
If $A$ is closed and mapping negligible, we call it an $\Omega$\emph{-set}; if
it is a closed homotopy negligible, we call it a $\mathcal{Z}$\emph{-set}. One
of the beauties of homotopy negligible sets (and thus $\mathcal{Z}$-sets) is
captured by the following simple observation.

\begin{proposition}
\label{Proposition: homotopy negligible sets}If $A$ is a homotopy negligible
in $Y$, then $s:Y-A\hookrightarrow Y$ is a homotopy equivalence.
\end{proposition}

\begin{proof}
Let $H:Y\times\left[  0,1\right]  \rightarrow Y$ be a homotopy with
$H_{0}=\operatorname*{id}_{Y}$ and $H_{t}\left(  Y\right)  \subseteq Y-A$ for
all $t>0$. Define $f:Y\rightarrow Y-A$ by $f\left(  x\right)  =H_{1}\left(
x\right)  $. Then $s\circ f=H_{1}$ which is homotopic to $\operatorname*{id}%
_{Y}$ by assumption. Similarly, $f\circ s=\left.  H_{1}\right\vert _{Y-A}$ can
be homotoped to $\operatorname*{id}_{Y-A}$ by applying $\left.  H\right\vert
_{\left(  Y-A\right)  \times\left[  0,1\right]  }$.
\end{proof}

\begin{example}
The homotopy negligible subsets of a compact manifold with boundary $M^{n}$
are precisely the subsets of $\partial M^{n}$; the $\mathcal{Z}$-sets are the
closed subsets of $\partial M^{n}$.
\end{example}

\begin{example}
For the space $\overline{X}=X\sqcup R$, described in Example
\ref{Example: cone on a sequence 1}, $R$ is mapping negligible but not an
$\Omega$-set. It is not homotopy negligible even though the conclusion of
Proposition \ref{Proposition: homotopy negligible sets} holds.
\end{example}

\begin{example}
For the space $\overline{X}=X\sqcup R$, described in Example
\ref{Example: compactifications of a ray}\ref{Subexample: sin(1/x) curve}, $R$
is an $\Omega$-set but not a $\mathcal{Z}$-set, and the conclusion of
Proposition \ref{Proposition: homotopy negligible sets} fails.
\end{example}

A metric compactification $\overline{X}=X\sqcup R$ of a space $X$ is called:

\begin{itemize}
\item a \emph{mapping negligible compactification} if $R$ is mapping
negligible in $\overline{X}$;

\item an $\Omega$\emph{-compactification} if $R$ is an $\Omega$-set in
$\overline{X}$;

\item a \emph{homotopy negligible compactification} if $R$ is homotopy
negligible in $\overline{X}$; and

\item a $\mathcal{Z}$\emph{-compactification} if $R$ is a $\mathcal{Z}$-set.
in $\overline{X}$.
\end{itemize}

All compactifications in this paper are obtained via the strategy described in
Section \ref{Subsection: Inverse and direct limits associated with X}. The
question then becomes: When does that procedure yield a compactification
satisfying one (or more) of the above definitions?

\subsection{Mapping negligible compactifications and $\Omega$-compactifications}

Our most general compactification theorem is the following. In the proof,
$\simeq$ denotes homotopic maps, and a composition of bonding maps
$r_{n-1}\circ\cdots\circ r_{m}:C_{m}\rightarrow C_{n}$ is denoted $r_{n}^{m}$.

\begin{theorem}
\label{Theorem: General compactification theorem}Let $C_{0}\overset{r_{1}%
}{\longleftarrow}C_{1}\overset{r_{2}}{\longleftarrow}C_{2}\overset{r_{3}%
}{\longleftarrow}\cdots$ be an inverse sequence of compact metric spaces and
retractions and let $X_{I}=%
{\textstyle\bigcup_{i=0}^{\infty}}
C_{i}$ endowed with the topology $\mathcal{T}_{I}$. Adopt all other notational
conventions established in Section \ref{Section: The general setting}. Then

\begin{enumerate}
\item $\underleftarrow{\lim}\left\{  C_{i},r_{i}\right\}  $ is a mapping
negligible compactification of $e\left(  X\right)  $,

\item equivalently, $\overline{X_{I}}=X_{I}\sqcup R$ is a mapping negligible
compactification of $X_{I}$,

\item if the conditions in Theorem \ref{Theorem: TFAE with T_I = T_D} hold,
then these are $\Omega$-compactifications, and

\item if each $r_{i}$ is a homotopy equivalence, then the projection maps
$\pi_{j}:\underleftarrow{\lim}\left\{  C_{i},r_{i}\right\}  \rightarrow C_{j}$
and their analogs $\pi_{j}^{\prime}:\overline{X_{I}}\rightarrow C_{j}$ are
shape equivalences.
\end{enumerate}
\end{theorem}

\begin{proof}
For the first assertion, apply $e\circ\pi_{i}$ for arbitrarily large $i$; the
second assertion follows immediately. The third is an application of
\ref{Proposition: local compactness of X}.

The final assertion is a routine application of the inverse sequence approach
to shape theory, if the $C_{i}$ are assumed to be finite polyhedra---or even
just ANRs. In particular, $C_{0}\overset{r_{1}}{\longleftarrow}C_{1}%
\overset{r_{2}}{\longleftarrow}C_{2}\overset{r_{3}}{\longleftarrow}\cdots$
then represents the shape of $\underleftarrow{\lim}\left\{  C_{i}%
,r_{i}\right\}  $, and is pro-homotopic to each constant inverse sequence
$C_{j}\overset{\operatorname*{id}}{\longleftarrow}C_{j}%
\overset{\operatorname*{id}}{\longleftarrow}\cdots$. For a brief introduction
to that approach, see \cite[\S 3.7]{Gui16}. If the $C_{i}$ are permitted to be
arbitrary compacta, a more subtle approach is needed. The following argument
was communicated to us by Jerzy Dydak. It uses the approach to shapes which
considers sets of maps into ANRs.

By \cite[Th.4.1.5]{DySe78} (definition on p.22), $\{C_{i},r_{i}\}$ satisfies
the following \textquotedblleft continuity condition\textquotedblright:\ Any
map $f:\underleftarrow{\lim}\left\{  C_{i},r_{i}\right\}  \rightarrow K$,
where $K$ is an ANR, factors up to homotopy as $f\simeq g\circ\pi_{n}$, where
$g:C_{n}\rightarrow K$, for some $n$. And, if $g\circ\pi_{n}\simeq h\circ
\pi_{n}$, then there is $m>n$ such that $g\circ r_{n}^{m}\simeq h\circ
r_{n}^{m}$.

To show that $\pi_{0}:\underleftarrow{\lim}\left\{  C_{i},r_{i}\right\}
\rightarrow C_{0}$ is a shape equivalence, we need to show that every map
$f:\underleftarrow{\lim}\left\{  C_{i},r_{i}\right\}  \rightarrow K$ factors
uniquely, up to homotopy, via $\pi_{0}$. It factors as $g\circ\pi_{n}$ for
some $n\geq0$, so the homotopy inverse $s:C_{0}\rightarrow C_{n}$ of
$r_{0}^{n}$ satisfies $s\circ\pi_{0}\simeq\pi_{n}$ (as $r_{0}^{n}\circ\pi
_{n}=\pi_{0}$), and therefore $f\simeq\left(  g\circ s\right)  \circ\pi_{0}$.

For uniqueness, suppose $g,h:C_{0}\rightarrow K$ satisfy $g\circ\pi_{0}\simeq
h\circ\pi_{0}$. Then for some $n>1$, $g\circ r_{0}^{n}\simeq h\circ r_{0}^{n}$
and composing with $s$ gives $g\simeq h$. The same argument applies to
$\pi_{i}$ for $i>0$.
\end{proof}

\begin{remark}
Section \ref{Section: Examples} contains several examples where the projection
maps are homotopy equivalences and others where they are shape equivalences
but not homotopy equivalences.
\end{remark}

\subsection{Homotopy negligible compactifications and $\mathcal{Z}%
$-compactifications\label{Subsection: Homotopy negligible compactifications}}

We are ready to prove one of our main theorems.

\begin{theorem}
\label{Theorem: homotopy negligible compactifications}Let $C_{0}%
\overset{\rho_{1}}{\swarrow}C_{1}\overset{\rho_{2}}{\swarrow}C_{2}%
\overset{\rho_{3}}{\swarrow}\cdots$ is an inverse sequence of compact metric
spaces and compact collapses, and let $X_{I}=%
{\textstyle\bigcup_{i=0}^{\infty}}
C_{i}$ endowed with the topology $\mathcal{T}_{I}$. Adopt all other notational
conventions established in Section \ref{Section: The general setting}. Then
there is a homotopy $H:\underleftarrow{\lim}\left\{  C_{i},\rho_{i}\right\}
\times\left[  0,1\right]  \rightarrow\underleftarrow{\lim}\left\{  C_{i}%
,\rho_{i}\right\}  $ such that $H_{t}\left(  \underleftarrow{\lim}\left\{
C_{i},\rho_{i}\right\}  \right)  \subseteq e\left(  X\right)  $ for all $t>0$;
moreover, for each $i\geq0$, $\left.  H\right\vert _{\underleftarrow{\lim
}\left\{  C_{i},\rho_{i}\right\}  \times\left[  0,\frac{1}{2^{i}}\right]  }$
is a compact collapse of $\underleftarrow{\lim}\left\{  C_{i},\rho
_{i}\right\}  $ onto $e\left(  C_{i}\right)  $. By letting $J=\overline
{e}^{-1}\circ H\circ\left(  \overline{e},\operatorname*{id}\right)  $, it
follows that

\begin{enumerate}
\item $\overline{X_{I}}=X_{I}\sqcup R$ is a homotopy negligible
compactification of $X_{I}$,

\item $\left.  J\right\vert _{\overline{X}\times\left[  0,\frac{1}{2^{i}%
}\right]  }$ is a compact collapse of $\overline{X_{I}}$ onto $C_{i}$ for each
$i\geq0$,

\item $X_{I}$ deformation retracts onto $C_{i}$ for each $i\geq0$, and

\item if the conditions in Theorem \ref{Theorem: TFAE with T_I = T_D} hold,
then $\overline{X_{I}}$ is a $\mathcal{Z}$-compactification,
\end{enumerate}
\end{theorem}

\begin{proof}
For each $i\geq1$, let $\phi_{i}:C_{i}\times\left[  0,1\right]  \rightarrow
C_{i}$ be a track-faithful deformation retraction of $C_{i}$ onto $C_{i-1}$
culminating in $\rho_{i}$. Then define $H_{0}:C_{0}\times\left[  0,1\right]
\rightarrow C_{0}$ to be the projection map and, for $i>0$, inductively define
$H_{i}:C_{i}\times\left[  0,1\right]  \rightarrow C_{i}$ by%
\[
H_{i}\left(  x,t\right)  =\left\{
\begin{tabular}
[c]{lll}%
$x$ & if & $0\leq t\leq1/2^{i}$\\
$\phi_{i}\left(  x,2^{i}t-1\right)  $ & if & $1/2^{i}\leq t\leq1/2^{i-1}$\\
$H_{i-1}\left(  \rho_{i}\left(  x\right)  ,t\right)  $ & if & $1/2^{i-1}\leq
t\leq1$%
\end{tabular}
\ \right.
\]
Using the track-faithful property of the $\phi_{i}$, one can check that the
following diagram commutes
\[%
\begin{array}
[c]{cccccccc}%
C_{0}\times\left[  0,1\right]   & \overset{\rho_{1}\times\operatorname*{id}%
}{\longleftarrow} & C_{1}\times\left[  0,1\right]   & \overset{\rho_{2}%
\times\operatorname*{id}}{\longleftarrow} & C_{2}\times\left[  0,1\right]   &
\overset{\rho_{3}\times\operatorname*{id}}{\longleftarrow} & C_{3}%
\times\left[  0,1\right]   & \overset{\rho_{4}\times\operatorname*{id}%
}{\longleftarrow}\cdots\\
H_{0}\downarrow\quad &  & H_{1}\downarrow\quad &  & H_{2}\downarrow\quad &  &
H_{3}\downarrow\quad & \\
C_{0} & \overset{\rho_{1}}{\longleftarrow} & C_{1} & \overset{\rho
_{2}}{\longleftarrow} & C_{2} & \overset{\rho_{3}}{\longleftarrow} & C_{3} &
\overset{\rho_{4}}{\longleftarrow}\cdots
\end{array}
\]

It therefore induces a map $H$ between the inverse limits of the top and
bottom rows. Since each bonding map in the top row is a product, the inverse
limit of that sequence factors as a product of inverse limits; and since the
second factor in each of those product maps is $\operatorname*{id}_{\left[
0,1\right]  }$, the second inverse limit is canonically homeomorphic to
$\left[  0,1\right]  $. So we have a homotopy
\[
H:\underleftarrow{\lim}\left\{  C_{i},\rho_{i}\right\}  \times\left[
0,1\right]  \rightarrow\underleftarrow{\lim}\left\{  C_{i},\rho_{i}\right\}
\]
which, by definition, is the restriction of $H_{0}\times H_{1}\times
H_{2}\times\cdots:\prod C_{i}\times\left[  0,1\right]  \rightarrow$ $\prod
C_{i}$.

To see that $H$ satisfies the desired properties, recall that, for each
$i\geq0$, $e\left(  C_{i}\right)  $ is the set of all points $\left(
x_{0},x_{1},x_{2},\cdots\right)  $ in $\underleftarrow{\lim}\left\{
C_{i},\rho_{i}\right\}  $ for which $x_{j}=x_{i}$ for all $j\geq i$ (see
Theorem \ref{Theorem: e is an embedding into the inverse limit}). For $t>0$,
choose $j$ such that $1/2^{j}\leq t\leq1/2^{j-1}$ and notice that, for $i\geq
j$ and $\left(  x_{0},x_{1},x_{2},\cdots\right)  \in\underleftarrow{\lim
}\left\{  C_{i},\rho_{i}\right\}  $
\[
H_{i}\left(  x_{i},t\right)  =H_{i}\left(  \rho_{i+1}\left(  x_{i+1}\right)
,t\right)  =H_{i+1}\left(  x_{i+1},t\right)
\]
so $H\left(  \left(  x_{0},x_{1},x_{2},\cdots\right)  ,t\right)  \in e\left(
C_{i}\right)  \subseteq e\left(  X\right)  $. Furthermore, if $\left(
x_{0},x_{1},x_{2},\cdots\right)  \in e\left(  C_{i}\right)  $, and
$t\leq1/2^{i}$ then $H_{i}\left(  x_{i},t\right)  =x_{i}$, so by the same
formula, $H_{j}\left(  x_{i},t\right)  =x_{i}$ for all $j>i$. Thus, $\left.
H\right\vert _{\underleftarrow{\lim}\left\{  C_{i},\rho_{i}\right\}
\times\left[  0,\frac{1}{2^{i}}\right]  }$ is a deformation retract onto
$e\left(  C_{i}\right)  $. A similar calculation shows that this homotopy is
track faithful.
\end{proof}

\begin{remark}
The above proof is a generalization and refinement of an argument outlined by
Jeremy Brazus in his \emph{Wild Topology }blog. See \cite{Bra19}.
\end{remark}

\subsection{Compactifications of ANRs}

In the context of geometric group theory and algebraic and geometric topology,
the spaces $Y$, $X$ and $\overline{X}$ discussed at the beginning of this
section are frequently ANRs. When that is the case, several stronger
assertions and equivalences hold. It will be useful to review a few of those.

A metric space $W$ is an \emph{absolute neighborhood retract }(ANR) if,
whenever it is embedded as a closed subset of a metric space $S$, a
neighborhood of $W$ in $S$ retracts onto $W$. We call $W$ an \emph{absolute
retract} (AR) if whenever $W$ is embedded as a closed subset of a metric space
$S$, the entire space $S$ retracts onto $W$. Equivalently, an AR is a
contractible ANR. Notice that we are requiring $W$ and $S$ to be metric spaces
in these definitions; not all treatments of ANR theory include those assumptions.

Key examples of ANRs include: manifolds; metric CW and simplicial complexes
(in particular, all locally finite complexes under the Whitehead topology);
and CAT(0) spaces. See \cite{Hu65} and \cite{Ont05}. The following two well
known propositions are useful for identifying larger classes of ANRs.

\begin{proposition}
A finite-dimensional metric space $W$ is an ANR if and only if it is locally
contractible; in other words, for every $x\in W$ and neighborhood $U$ of $x$,
there exists a neighborhood $V$ of $x$ that contracts in $U$.
\end{proposition}

\begin{proposition}
\label{Proposition: being an ANR is a local property}Being an ANR is a local
property. More precisely, every open subset of an ANR is and ANR, and a metric
space with the property that every point has an ANR neighborhood is an ANR.
\end{proposition}

\begin{corollary}
In order for a locally compact space $X$ to admit a compactification
$\overline{X}=X\sqcup R$, in which $\overline{X}$ is an ANR, it is necessary
for $X$ to be an ANR.
\end{corollary}

Versions of the following are woven into the literature in places where
$\mathcal{Z}$-compactifications are used to define and analyze group
boundaries. For example, \cite{Bes96}.

\begin{proposition}
\label{Proposition: equivalent to being a Z-set}For a closed subset $A$ of a
compact ANR $Y$, the following are equivalent:

\begin{enumerate}
\item $A$ is a $\mathcal{Z}$-set,

\item $A$ is an $\Omega$-set,

\item for every open set $U\subseteq Y$, the inclusion $U-A\hookrightarrow U$
is a homotopy equivalence.
\end{enumerate}
\end{proposition}

\begin{proof}
All of these can be obtained from Toru\'{n}czyk's Equivalence Theorem from
\cite{Anc23}. See also \cite{Tor78}.
\end{proof}

In order to apply Proposition \ref{Proposition: equivalent to being a Z-set},
the following partial converse to Proposition
\ref{Proposition: being an ANR is a local property} and its corollary can be useful.

\begin{proposition}
\label{Proposition: homotopy negligible compactifications of ANRs}If $A$ is a
homotopy negligible subset of a compact metric space $Y$, and $Y-A$ is an ANR,
then $Y$ is an ANR. In other words, every homotopy negligible compactification
of an ANR is an ANR.
\end{proposition}

\begin{proof}
Apply Hanner's Theorem \cite{Han51}, as described in \cite{Tir11}.
\end{proof}

\begin{corollary}
\label{Corollary: compactifications of ANRs}Suppose $\overline{X}=X\sqcup R$
is an $\Omega$-compactification of an ANR [resp., AR] $X$. Then $\overline{X}$
is a $\mathcal{Z}$-compactification of $X$ if and only if $\overline{X}$ is an
ANR [resp., AR].
\end{corollary}

\begin{proof}
The forward and reverse implications follow from Propositions
\ref{Proposition: homotopy negligible compactifications of ANRs} and
\ref{Proposition: equivalent to being a Z-set}, respectively. The assertion
about ARs relies on Proposition \ref{Proposition: homotopy negligible sets}.
\end{proof}

By using essentially the same argument, one can deduce a strengthening of
Proposition \ref{Proposition: homotopy negligible compactifications of ANRs},
that will be useful in the next section.

\begin{proposition}
\label{Proposition: strengthening of the Hanner application}Let $A$ is a
homotopy negligible subset of a compact metric space $Y$ and $H:Y\times\left[
0,1\right]  \rightarrow Y$ a corresponding homotopy with $H_{0}%
=\operatorname*{id}_{Y}$ and $H_{t}\left(  Y\right)  \subseteq Y-A$ for all
$t>0$. Suppose that for each $\varepsilon>0$, there exist a $t\in\left[
0,\varepsilon\right]  $ and an ANR $B_{t}\subseteq Y$ such that $H_{t}\left(
Y\right)  \subseteq B_{t}$, then $Y$ is an ANR.
\end{proposition}

\begin{corollary}
\label{Corollary: more general ANR conclusion}If $\left\{  C_{i},\rho
_{i}\right\}  $ is an inverse sequence of compact ANRs and compact collapses,
and $\overline{X}=X_{I}\sqcup R$ is the corresponding homotopy negligible
compactification of $X_{I}$ (as in Theorem
\ref{Theorem: homotopy negligible compactifications}), then $\overline{X}$ is
an ANR.
\end{corollary}

\section{Applications and more examples}

We conclude this article with a variety of theorems and examples that can be
proved or reinterpreted using the ideas presented here.

\subsection{General applications to simplicial and CW complexes and mapping
telescopes}

The first pair of applications are immediate consequences of Theorem
\ref{Theorem: homotopy negligible compactifications} and the discussion in
Sections \ref{Section: The general setting} and
\ref{Section: Expansions and collapses in the simplicial category}.

\begin{theorem}
Let $K$ be a simplicial complex that collapses to a finite subcomplex $L$, and
let $\left\vert K\right\vert _{I}$ be the geometric realization of $K$ with
the topology $\mathcal{T}_{I}$ induced by a corresponding inverse sequence of
simplicial collapses beginning with $\left\vert L\right\vert $. Then the
compactification $\overline{\left\vert K\right\vert _{I}}=\left\vert
K\right\vert _{I}\sqcup R$ is homotopy negligible. Furthermore, there is a
compact collapse of $\overline{\left\vert K\right\vert _{I}}$ onto $\left\vert
L\right\vert $ and a deformation retraction of $\left\vert K\right\vert _{I}$
to $\left\vert L\right\vert $.
\end{theorem}

\begin{corollary}
\label{Corollary: Z-compatifications of simplicial complexes}Let $K$ be a
locally finite simplicial complex that collapses to a finite subcomplex $L$,
and let $\left\vert K\right\vert $ be the geometric realization of $K$ with
its standard topology. Then there is a $\mathcal{Z}$-compactification
$\overline{\left\vert K\right\vert }=\left\vert K\right\vert \sqcup R$ such
that $\overline{\left\vert K\right\vert }$ is an ANR. Furthermore, there is a
compact collapse of $\overline{\left\vert K\right\vert }$ onto $\left\vert
L\right\vert $ and a deformation retraction of $\left\vert K\right\vert $ to
$\left\vert L\right\vert $. If $\left\vert L\right\vert $ is contractible,
then $\left\vert K\right\vert $ and $\overline{\left\vert K\right\vert }$ are ARs.
\end{corollary}

\begin{proof}
Along with the previous theorem, apply Propositions
\ref{Proposition: local compactness of X} and
\ref{Proposition: equivalence of the topologies} and Corollary
\ref{Corollary: compactifications of ANRs}.
\end{proof}

\begin{remark}
Examples \ref{Example: cone on a sequence 2} and
\ref{Example: compactifications of a ray} show how $\mathcal{T}_{I}$ can
differ from the usual topologies on $\left\vert K\right\vert $, even when $K$
is locally finite. Examples \ref{Example: cone on a sequence 1} and
\ref{Example: compactifications of a ray} illustrate the importance of having
bonding maps that are compact collapses and not just arbitrary deformation retractions.
\end{remark}

By combining the above argument with Corollary \ref{Corollary: CW collapse},
we can obtain a version of Corollary
\ref{Corollary: Z-compatifications of simplicial complexes} for CW complexes.

\begin{theorem}
Let $X$ be a locally finite CW complex that collapses, via elementary CW
collapses, to a finite subcomplex $L$. Then there exists a $\mathcal{Z}%
$-compactification $\overline{X}=X\sqcup R$ such that $\overline{X}$ is an
ANR. Furthermore, there is a compact collapse of $\overline{X}$ onto $L$ and a
deformation retraction of $X$ to $L$. If $L$ is contractible, then $X$ and
$\overline{X}$ are ARs.
\end{theorem}

Versions of the next theorem have played a key role in earlier work on
$\mathcal{Z}$-compactifications. See \cite{ChSi76}.

\begin{theorem}
Let $K_{0}\overset{f_{1}}{\longleftarrow}K_{1}\overset{f_{2}}{\longleftarrow
}K_{2}\overset{f_{3}}{\longleftarrow}\cdots$ be an inverse sequence of compact
metric spaces. Then there is a $\mathcal{Z}$-compactification
\[
\overline{\operatorname*{Tel}\left(  K_{i},f_{i}\right)  }=\operatorname*{Tel}%
\left(  K_{i},f_{i}\right)  \sqcup R
\]
where $R=\underleftarrow{\lim}(\left\{  K_{i},f_{i}\right\}  )$. The spaces
$\overline{\operatorname*{Tel}\left(  K_{i},f_{i}\right)  }$ and
$\operatorname*{Tel}\left(  K_{i},f_{i}\right)  $ are ANRs if and only if each
$K_{i}$ is an ANR.
\end{theorem}

\begin{proof}
The main assertion combines Examples \ref{Example: Mapping telescope} and
\ref{Example: mapping cylinder collapse} and Theorem
\ref{Theorem: homotopy negligible compactifications}. For the follow-up
assertion, apply \cite[IV.1]{Hu65} and Corollary
\ref{Corollary: more general ANR conclusion} to see that $\operatorname*{Tel}%
\left(  K_{i},f_{i}\right)  $ and$\overline{\operatorname*{Tel}\left(
K_{i},f_{i}\right)  }$ are ANRs when all $K_{i}$ are ANRs. For the converse,
note that for each $i$, $\operatorname*{Tel}\left(  K_{i},f_{i}\right)  $
contains an open set that retract onto $K_{i}$.
\end{proof}

\subsection{Applications to CAT(0)
spaces\label{Subsection: Applications to CAT(0) spaces}}

Our first theorem about CAT(0) spaces is well known and obtainable using
analytic techniques. The methods developed in this paper provide an
alternative approach.

\begin{theorem}
\label{Theorem: Z-compactification of CAT(0) spaces}Every proper CAT(0) space
$Y$ is $\mathcal{Z}$-compactifiable by addition of the visual sphere at
infinity. Moreover, $\overline{Y}=Y\sqcup\partial_{\infty}Y$ is an AR.
\end{theorem}

\begin{proof}
The basic properties of this compactification were covered in Example
\ref{Example: CAT(0) spaces}. The assertions that this is a $\mathcal{Z}%
$-compactification and that $\overline{Y}$ is an AR follow from
\ref{Theorem: homotopy negligible compactifications} and Propositions
\ref{Proposition: homotopy negligible compactifications of ANRs}.
\end{proof}

\begin{theorem}
[\cite{Gul23}]\label{Theorem: cubical Z-compactification}Every locally finite
CAT(0) cube complex is collapsible via a sequence of elementary cubical
collapses and therefore admits a $\mathcal{Z}$-compactification $\overline
{Y}=Y\sqcup\partial_{\square}Y$ corresponding to such a sequence.
\end{theorem}

\begin{remark}
The $\mathcal{Z}$-compactification in Theorem
\ref{Theorem: cubical Z-compactification} differs from the one in Theorem
\ref{Theorem: Z-compactification of CAT(0) spaces}. The proof in \cite{Gul23}
is combinatorial rather than analytic, and the resulting boundary (referred to
as a \emph{cubical boundary}) retains some of the combinatorial structure of
$Y$. Further investigation of that compactification is part of an ongoing
project by the authors of this paper.

An interesting sort of partial converse to Theorem
\ref{Theorem: cubical Z-compactification} is contained in \cite{AdFu21}. There
it is shown that every collapsible locally finite simplicial complex admits a
subdivision that is a CAT(0) cube complex.
\end{remark}

\subsection{Applications to manifolds}

Finite sequences of simplicial and CW expansions and collapses play important
roles in the study of compact manifolds. See \cite{RoSa82}. The
generalizations developed here can be used to obtain some analogous theorems
for noncompact manifolds; we offer a sample in this section. For the sake of
brevity, we focus our attention to contractible open manifolds.

A countable simplicial complex $K$ (or its geometric realization $\left\vert
K\right\vert $) is a \emph{combinatorial n-manifold} if the link of each
$k$-simplex of $K$ is PL homeomorphic to either $\left\vert \partial
\sigma^{n-k}\right\vert $ or $\left\vert \sigma^{n-k-1}\right\vert $. The
simplices with links of the latter type constitute a subcomplex $\partial
K\leq K$. Under these circumstances, $\left\vert K\right\vert $ satisfies the
topological definition of an $n$-manifold with $\partial\left\vert
K\right\vert =\left\vert \partial K\right\vert $. A \emph{PL n-manifold} is a
space $M^{n}$ together with a triangulation $h:\left\vert K\right\vert
\rightarrow M^{n}$ by a combinatorial $n$-manifold; we call this a \emph{PL
triangulation} of $M^{n}$.

A homeomorphism $h:\left\vert K\right\vert \rightarrow M^{n}$, where $M^{n}$
is a topological $n$-manifold and $K$ is \emph{not} a combinatorial
$n$-manifold, is called a \emph{non-PL triangulation }of $M^{n}$. The
existence of non-PL triangulations of manifolds is a deep fact first proved by
Edwards \cite{Edw06}. That work implies that all PL manifolds of dimension
$\geq5$ admit non-PL triangulations. By contrast, it is known that non-PL
triangulations do not exist in dimensions $\leq4$.

Our first theorem is an easy one.

\begin{theorem}
\label{Theorem: collapsing R^n}For all $n$, $%
\mathbb{R}
^{n}$ admits a simplicially collapsible PL triangulation.
\end{theorem}

\begin{proof}
Begin with an \emph{n}-simplex $\sigma^{n}$, then inductively add PL
triangulations of $\left\vert \partial\sigma^{n}\right\vert \times\left[
k,k+1\right]  $ of the sort described in Example
\ref{Example: simplicially collapsing a product}.
\end{proof}

The converse of Theorem \ref{Theorem: collapsing R^n} is more striking. It can
be viewed as the noncompact analog of an important theorem which asserts that
every compact simplicially collapsible combinatorial $n$-manifold is a PL
$n$-ball \cite[Cor.3.28]{RoSa82}.

\begin{theorem}
\label{Theorem: collapsible implies R^n}If a combinatorial $n$-manifold
$\left\vert K\right\vert $ is simplicially collapsible, then $\left\vert
K\right\vert $ is PL homeomorphic to $%
\mathbb{R}
^{n}$.
\end{theorem}

\begin{proof}
Let $\left\{  v\right\}  =L_{0}\overset{e}{\swarrow}L_{1}\overset{e}{\swarrow
}L_{2}\overset{e}{\swarrow}\cdots$ be a corresponding sequence of elementary
combinatorial expansions through subcomplexes such that $\cup L_{i}=K$. Then
each $L_{i}$ is finite and collapsible, so a regular neighborhood $N_{i}$ of
$\left\vert L_{i}\right\vert $ in $\left\vert K\right\vert $ is a PL $n$-ball.
By choosing wisely (or passing to a subsequence) we can arrange that
$N_{i}\subseteq\operatorname*{Int}N_{i+1}$ for all $i$, so by the
Combinatorial Annulus Theorem \cite[Cor.3.19]{RoSa82}, $N_{i+1}%
-\operatorname*{Int}N_{i}$ is PL homeomorphic to $S^{n-1}\times\left[
i,i+1\right]  $. Gluing these products together and plugging the hole with
$N_{0}$ gives a homeomorphism to $%
\mathbb{R}
^{n}$.
\end{proof}

\begin{corollary}
No triangulation of the Whitehead contractible 3-manifold is simplicially
collapsible. More generally, the only open manifolds of dimension $\leq4$ that
admit simplicially collapsible triangulations are those PL homeomorphic to $%
\mathbb{R}
^{n}$.
\end{corollary}

\begin{proof}
Every triangulation of a manifold of dimension $\leq4$ is PL, so the assertion
follows from Theorem \ref{Theorem: collapsible implies R^n}.
\end{proof}

If we allow non-PL triangulations, a different picture emerges, even for the
simplest class of contractible open manifolds.

\begin{theorem}
\label{Theorem: interiors of contractible manifolds are collapsible}If $C^{n}$
is a compact contractible PL\ $n$-manifold with $n\geq5$, then
$\operatorname*{int}C^{n}$ admits a simplicially collapsible non-PL triangulation.
\end{theorem}

We begin this proof with a lemma.

\begin{lemma}
\label{Lemma: compact contractible manifolds}Let $C^{n}$ be a compact
contractible $n$-manifold that is not homeomorphic to a ball. Then

\begin{enumerate}
\item \label{Assertion 1 of compact contractible manifolds}there does not
exist a PL triangulation of $C^{n}$ that is simplicially collapsible,

\item \label{Assertion 2 of compact contractible manifolds}there exists a
non-PL triangulation of $C^{n}$ that is simplicially collapsible if and only
if $n\geq5$ and $\partial C^{n}$ admits a PL triangulation. (The latter
condition holds whenever $n\geq6$.)

\item \label{Assertion 3 of compact contractible manifolds}$C^{n}$ is always
compactly collapsible.
\end{enumerate}
\end{lemma}

\begin{proof}
Assertion \ref{Assertion 1 of compact contractible manifolds} is a consequence
of regular neighborhood theory. Suppose such a triangulation exists, and let
$DC^{n}$ be the double of $C^{n}$ along its boundary. Then, a regular
neighborhood of $C^{n}$ in $DC^{n}$ is PL homeomorphic to $C^{n}$, but a
regular neighborhood of a compact collapsible polyhedron is always\ a PL ball.

Assertion \ref{Assertion 2 of compact contractible manifolds} relies on a
construction described in \cite{ADG97} which triangulates $C^{n}$ with a
simplicial complex of the form%
\[
T=(\left\{  p_{1}\right\}  \ast K_{1})\cup_{\left\{  q\right\}  \ast
J}(\left\{  p_{2}\right\}  \ast K_{2})
\]
(the union of two simplicial cones on finite simplicial complexes $K_{1}$ and
$K_{2}$ along a common conical subcomplex $\left\{  q\right\}  \ast J$ of
$K_{1}$ and $K_{2}$). Since every finite simplicial cone collapses to the
subcone over any of its subcomplexes, we can collapse $T$ to the subcomplex
$\left\{  p_{1},p_{2}\right\}  \ast\left(  \left\{  q\right\}  \ast J\right)
$ which, by associativity of simplicial joins is the cone $\left\{  q\right\}
\ast(\left\{  p_{1},p_{2}\right\}  \ast J)$. From there we can collapse to $q$.

For the parenthetic remark, note that the boundary of a compact contractible
$n$-manifold has the homology of an $\left(  n-1\right)  $-sphere. The
Kirby-Siebenmann obstruction \cite{KiSi77} to a PL triangulation of $\partial
C^{n}$ lies in $H^{4}\left(  \partial C^{n};%
\mathbb{Z}
_{2}\right)  $, and therefore vanishes when $n\geq6$.

For assertion \ref{Assertion 3 of compact contractible manifolds} choose any
contraction of $C^{n}$ that fixes a point. That contraction is a compact collapse.
\end{proof}

\begin{proof}
[Proof of Theorem
\ref{Theorem: interiors of contractible manifolds are collapsible}]If we
attach an exterior open collar $\partial C^{n}\times\lbrack0,\infty)$ to
$C^{n}$, the resulting space is homeomorphic to $\operatorname*{int}C^{n}$.
Thus,
\begin{equation}
\operatorname*{int}C^{n}\approx C^{n}\cup(\partial C^{n}\times\left[
0,1\right]  )\cup(\partial C^{n}\times\left[  1,2\right]  )\cup(\partial
C^{n}\times\left[  2,3\right]  )\cup\cdots
\label{Item: collar decomposition of C^n}%
\end{equation}
By Lemma \ref{Lemma: compact contractible manifolds} we can choose a
simplicially collapsible non-PL triangulation $\left\vert K\right\vert $ of
$C^{n}$ with a subcomplex $L\leq K$ that triangulates $\partial C^{n}$. A
sequence of applications of Example
\ref{Example: examples of compact collapses} applied to $\left\vert
L\right\vert \times\left[  i,i+1\right]  $ produces triangulations of
$\partial C^{n}\times\left[  i,i+1\right]  $ which agree on their overlap and
collapse to their left-hand boundaries. Assembling these produces the desired triangulation.
\end{proof}

If we allow arbitrary compact collapses, a still different picture emerges.

\begin{theorem}
The interior of every compact contractible $n$-manifold $C^{n}$ admits a
compact collapse.
\end{theorem}

\begin{proof}
Begin with the decomposition of $\operatorname*{int}C^{n}$ provided in
(\ref{Item: collar decomposition of C^n}). Since each $\partial C^{n}%
\times\left[  i,i+1\right]  $ compactly collapses to $\partial C^{n}%
\times\left\{  i\right\}  $ (Example
\ref{Example: elementary simplicial collapses}) and $C^{n}$ compactly
collapses to a point, the theorem follows.
\end{proof}

It is not the case that \emph{every} contractible open manifold is compactly collapsible.

\begin{theorem}
The Whitehead contractible 3-manifold $W^{3}$ is not compactly collapsible.
\end{theorem}

\begin{proof}
By Theorem \ref{Theorem: homotopy negligible compactifications}, if $W^{3}$
were compactly collapsible, it would be $\mathcal{Z}$-compactifiable. Then, by
\cite[Th.1.2]{GuTi03}, $W^{3}$ would be semistable at infinity, contradicting
\cite[Prop.16.4.2]{Geo08}.
\end{proof}

\begin{remark}
The Whitehead manifold is by no means unique. In all dimensions $\geq3$ there
exist contractible open manifolds that are not semistable; hence they are not
$\mathcal{Z}$-compactifiable; hence they are not compactly collapsible. (Note.
For non-manifolds, compact collapsibility does not imply semistability.)
\end{remark}

Another interesting collection of contractible open $n$-manifolds are the
exotic universal covers of closed aspherical $n$-manifolds constructed by
Davis in \cite{Dav83}. We will take a closer look at the collapsibility
properties of those manifolds in \cite{AGS23}. For now, we mention one
particularly nice class.

\begin{theorem}
Let $K$ be a flag simplicial complex for which $\left\vert K\right\vert $ is a
closed $\left(  n-1\right)  $-manifold with the same homology as $S^{n-1}$.
Then the right-angled Coxeter group $W_{K}$ with nerve $K$ acts geometrically
on a CAT(0) cube complex $X$ whose underlying space is a contractible open
$n$-manifold that is not homeomorphic to $%
\mathbb{R}
^{n}$. The space $X$ is cubically collapsible.
\end{theorem}

\begin{proof}
The space $X$ is constructed in \cite{ADG97}. It is a variation on the Davis
complex $\Sigma_{K}$ (which is not a manifold). By \cite{Gul23}, $X$ is
cubically collapsible.
\end{proof}

\subsection{Applications to 1-dimensional spaces}

\begin{theorem}
\label{Theorem: Z-compactifying finite valence trees}Every finite valence tree
$T$ is simplicially collapsible and admits a $\mathcal{Z}$-compactification
$\overline{T}=T\sqcup R$ where $\overline{T}$ is an AR and $R$ is homeomorphic
to the space of ends of $T$.
\end{theorem}

\begin{proof}
Endow $T$ with the path length metric; choose a base vertex $v_{0}$; and for
each $i\in%
\mathbb{N}
$, let $K_{i}$ be the closed ball of radius $i$ centered at $v_{0}$. Then each
$K_{i}$ is a finite subtree and $\overline{K_{i+1}-K_{i}}$ is a collection of
edges, each having one end point in $K_{i-1}$. Clearly there is a simplicial
collapse $\rho_{i+1}:K_{i+1}\rightarrow K_{i}$. Now apply Corollary
\ref{Corollary: Z-compatifications of simplicial complexes}. The
correspondence between $R$ and the space of ends is almost immediate from the
definitions. See, for example \cite[\S 3.3]{Gui16}.
\end{proof}

Provided $T$ is countable, we can prove a weaker result without assuming
finite valences.

\begin{theorem}
\label{Theorem: compactifying more general trees}Every countable tree $T$ is a
union of finite subtrees $\left\{  v_{0}\right\}  =T_{0}\subseteq
T_{1}\subseteq T_{2}\subseteq\cdots$ with the property that $\overline
{T_{i+1}-T_{i}}$ is a collection of edges, each having one end point in
$T_{i-1}$; thus $T$ is simplicially collapsible. Under the topology induced by
the corresponding inverse sequence, $T_{I}$ admits a homotopy negligible
compactification to an AR $\overline{T_{I}}=T_{I}\sqcup R$.
\end{theorem}

\begin{proof}
We leave it to the reader to prove the existence of the collection $\left\{
T_{i}\right\}  $. From there, the conclusions follow from Theorem
\ref{Theorem: homotopy negligible compactifications} and Corollary
\ref{Corollary: more general ANR conclusion}.
\end{proof}

\begin{remark}
Example \ref{Example: cone on a sequence 2} is an illustration of Theorem
\ref{Theorem: compactifying more general trees}. Note that $T_{I}$ is not the
usual topology on $T$. Nevertheless, this theorem provides a method for
defining and topologizing the ends of $T$. For example, the tree in Example
\ref{Example: cone on a sequence 2} has no ends.
\end{remark}

A \emph{dendrite} is a compact, connected, locally path connected metric space
$D$ that is uniquely arc-wise connected; in other words, for any pair of
distinct points $p,q\in D$, there is a unique embedded arc with end points $p$
and $q$. Dendrites are of central importance in the field of continuum theory.
They have also played an important role in the study of hyperbolic groups
\cite{Bow99}. A fundamental theorem from continuum theory provides a link to
the previous two theorems.

\begin{theorem}
[{\cite[Th.10.27]{Nad92}}]\label{Theorem: from Nadler}Every dendrite $D$ is
homeomorphic to the inverse limit of a sequence%
\[
\left\{  v_{0}\right\}  =T_{0}\overset{\rho_{1}}{\swarrow}T_{1}\overset{\rho
_{2}}{\swarrow}T_{2}\overset{\rho_{3}}{\swarrow}\cdots
\]
of finite trees where $\overline{T_{i+1}-T_{i}}$ is an arc with one end point
in $T_{i}$ and $\rho_{i+1}$ is the corresponding compact collapse.
\end{theorem}

In this generality, we cannot assume there are fixed simplicial structures on
the $T_{i}$ making them subcomplexes of $T_{j}$ for $j\geq i$ and making the
$\rho_{i}$. Nevertheless, the proof used for Theorem
\ref{Theorem: compactifying more general trees} is applicable and gives the
following important theorem about dendrites.

\begin{corollary}
Every dendrite $D$ is a compact AR; in particular, $D$ is contractible.
\end{corollary}

With additional care, one can say a little more. A point $p\in D$ is an
\emph{end point} if it has arbitrarily small neighborhoods whose boundaries
are one-point sets. Let $D^{\left[  1\right]  }$ denote the set of end points
of $D$.

\begin{theorem}
The set $D^{\left[  1\right]  }$ is a homotopy negligible in $D$.
\end{theorem}

\begin{proof}
One strategy for proving Theorem \ref{Theorem: from Nadler} is to begin with a
countable dense set $X=\left\{  x_{0},x_{1},x_{2},\cdots\right\}  \subseteq$
$D$, then inductively define \thinspace$T_{i+1}$ to be the union of
\thinspace$T_{i}$ with the arc $A_{i+1}$ connecting $x_{i+1}$ to $x_{0}$. It
is easy to see that $T_{i}\cap A_{i+1}$ is either $\left\{  x_{0}\right\}  $
or a subarc of $A_{i+1}$; hence $T_{i+1}$ can be given the structure of a
finite tree. (If $T_{i}=T_{i+1}$, and we can later omit it from the
sequence.). It is a simple fact that $D^{\left[  1\right]  }$ is nowhere dense
in $D$, so we can choose $X$ to be disjoint from $D$. In that case, an element
of $D^{\left[  1\right]  }$ is never the leaf of a $T_{i}$; and clearly, it
cannot be a nonleaf of a $T_{i}$. Since or compactification is of the form
$\overline{D}=\left(  \cup T_{i}\right)  \sqcup R$, then $D^{\left[  1\right]
}\subseteq R$.
\end{proof}

\bibliographystyle{amsalpha}
\bibliography{Biblio}

\end{document}